\documentclass[12pt]{article}
\usepackage{amssymb,amsfonts,amsmath,amsthm, mathtools, url}
\usepackage{tikz}
\usetikzlibrary{automata,positioning}
\usepackage{caption, subcaption}
\usepackage[square,numbers, compress]{natbib}
\usepackage[demo,
export]{adjustbox}
\newcommand{\inner}[2]{\langle#1,#2\rangle}

\newcommand{\shape}[1]{\mbox{\it Shape}{\left(#1\right)}}

\newcommand{\odin}{\mbox {\bf 1}}

\newcommand{\pp}{{\mathbb P}}

\newcommand{\nn}{{\mathbb N}}

\newcommand{\E}{{\mathbb E}} 

\newcommand{\rr}{{\mathbb R}}

\newcommand{\cala}{{\mathcal A}}
\newcommand{\calb}{{\mathcal B}}

\newcommand{\calw}{{\mathcal W}}
\newcommand{\calh}{{\mathcal H}}

\newcommand{\call}{{\mathcal L}}
\newcommand{\caln}{{\mathcal N}}

\newcommand{\calt}{{\mathcal T}}

\newcommand{\calm}{{\mathcal M}}

\newcommand{\veps}{\varepsilon}

\newcommand{\calv}{{\mathcal V}}

\newcommand{\diag}[1]{Diag\left(#1\right)}

\newcommand{\norm}[1]{\|#1\|}
\newcommand{\beq}{\begin{eqnarray*}}
	\newcommand{\feq}{\end{eqnarray*}}
\newcommand{\beqn}{\begin{eqnarray}}
	\newcommand{\feqn}{\end{eqnarray}}
\newcommand{\as}{\mbox{\rm a.\,s.}}

\newtheorem{theorem}{Theorem}
\makeatletter \@addtoreset{theorem}{section}\makeatother

\newtheorem{lemma}[theorem]{Lemma}

\newtheorem*{theorem*}{Theorem}
\newtheorem*{conj*}{Conjecture}

\newtheorem{proposition}[theorem]{Proposition}
\newtheorem{corollary}[theorem]{Corollary}
\newtheorem{remark}[theorem]{Remark}
\newtheorem{example}[theorem]{Example}
\usepackage{algorithm}
\usepackage{graphicx}
\usepackage[noend]{algpseudocode}
\makeatletter
\def\BState{\State\hskip-\ALG@thistlm}
\makeatother

\DeclareCaptionLabelFormat{noname}{#2}
\newlength\myindent
\title{Lyapunov Exponents for Sparsely Coupled Linear Cocycles}

\author{
	Reza~Rastegar\thanks{Woodinville, WA; e-mail:  reza.j.rastegar@gmail.com}
}

\begin{document}
	\maketitle
	
	\begin{abstract}
		This paper studies structured products of real matrices for which the top Lyapunov exponent can be accessed by reducing the dynamics to an amenable generalization of upper triangular matrices.  Exploiting prescribed zero patterns (including block-triangularity and sparse decompositions, conveniently encoded by a directed sparsity graph), we obtain explicit, computable bounds and, in favorable cases, formulas for $\gamma_1$ by combining deterministic triangular controls with a suitable refinement of the Furstenberg--Kifer lemma \cite{fifa} for block-triangular products.  The estimates apply both to tempered (possibly deterministic) sequences and to stationary ergodic random cocycles under standard integrability.  We also discuss applications to perturbation models for linear systems, including low-rank updates, where the reduction converts the problem to lower-dimensional or scalar cocycles.
	\end{abstract}
	
	
	\section{Introduction}

	Lyapunov exponents play a central role in the analysis of random and chaotic dynamical systems: they quantify sensitivity to initial conditions and, in the linear-cocycle setting, encode long-time growth rates of matrix products; see, e.g., \cite{arnold, barl, boug, fuks, viana}.  Despite their importance, explicit formulas are available only for a comparatively small collection of solvable ensembles---classically, many computations focus on $2\times 2$ products or on highly symmetric distributions (including several Gaussian-type models); see, for instance, \cite{rankone, com, comt, elf, fibbo, keyc, kiev, letac, lima, darm, mann, mark, pinkus} and \cite{adams, ahn, aker, cohen, fort, ens, kargin, newman}.  The purpose of this paper is to extend and reorganize a few ideas from \cite{rankone, fifa, kiev, pinkus} into a small toolkit for obtaining \emph{computable} bounds (and, in favorable cases, explicit values) for Lyapunov exponents of structured matrix products.  From a matrix-analysis viewpoint, the guiding principle is that triangular or block-triangular structure---and more generally prescribed sparsity/zero patterns with limited feedback---localize exponential growth to lower-dimensional diagonal dynamics, while the remaining entries contribute at most an explicitly controlled combinatorial factor.
	Concretely, our bounds are expressed in terms of diagonal-block products together with a directed graph encoding admissible off-diagonal couplings, and are therefore readily computable from a small collection of lower-dimensional norm estimates.
	The framework covers both stationary/ergodic matrix sequences and deterministic sequences satisfying the temperedness condition \eqref{xn}, so it can be read as a matrix-analytic toolkit for growth-rate bounds in structured or sparsely coupled linear systems.
	
	\par
	For an integer $d\ge 2$, let $\calm_d$ denote the space of $d\times d$ matrices with real-valued entries.  Throughout, we write $\|\cdot\|$ both for a generic norm on $\rr^d$ (when the particular choice is immaterial) and for the induced operator norm on $\calm_d$.  Let $\cala:=(A_n)_{n\in\nn}$ be a sequence in $\calm_d$ such that
	\beqn \label{xn}
	\limsup_{n\to\infty} \frac{1}{n}\log \norm{X_n} < \infty\qquad \mbox{\rm where} \qquad
	X_n := A_n \cdots A_1.
	\feqn
	We associate to $\cala$ the (upper) Lyapunov exponent function $\gamma:\rr^d\to\rr\cup\{-\infty\}$ defined by
	\beqn
	\label{ld}
	\gamma(v) := \limsup_{n\to\infty}  \frac{1}{n} \log \norm{X_n v}.
	\feqn
	In particular, we define the \emph{top (upper) Lyapunov exponent} of the sequence $\cala$ by
	\beqn \label{gamma1_def}
	\gamma_1(\cala)
	:= \limsup_{n\to\infty}\frac1n\log\|X_n\|
	= \sup_{\|v\|=1}\gamma(v).
	\feqn
	
	This directional growth rate is the basic object we aim to control in the structured regimes considered below.
	
	\par
	A particularly clean structure emerges when $\cala$ is random, stationary, and ergodic and satisfies the standard integrability condition
	\beqn \label{f_cond}
	\E\left( \log^+ \|A_1\| \right) < \infty,
	\feqn
	where, for $x>0$, $\log^+ x:=\max(0,\log x)$, with the convention $\log^+0:=0$.  Oseledets' multiplicative ergodic theorem then yields deterministically constant Lyapunov exponents
	\[
	\gamma_1>\cdots>\gamma_\ell,\qquad \ell\le d,
	\]
	together with a corresponding filtration of subspaces
	\[
	\rr^d = E_1 \supset E_2 \supset \cdots \supset E_{\ell} \supset E_{\ell+1}:=\{0\},
	\]
	such that the limsup in \eqref{ld} is in fact a limit and $\gamma(v)\in\{\gamma_1,\dots,\gamma_\ell\}$ almost surely, with the value determined by the position of $v$ in the filtration; see \cite[Section~2.1]{barr} and \cite{arnold, barl, viana}.  In particular, the exponents and subspaces are independent of the realization with probability one.
	
	\par
	Even in this ergodic framework, however, the exponents $\gamma_k$ are rarely available in closed form.  A notable exception occurs for triangular cocycles, where the diagonal entries dictate the asymptotic growth.  In particular, for upper triangular matrices we have:
	\begin{proposition}
		\label{pkus}
		\cite{pinkus}
		Let $\cala=(A_k)_{k\in\nn}\in\calm_d^\nn$ be a random, stationary, and ergodic sequence of $d\times d$ invertible matrices with $A_k(i,j) = 0$ for $i > j$ and such that \eqref{f_cond} holds. Then,
		\beq
		\gamma_1 = \max_{i\in [d]} \E \left( \log |A_1(i,i)|\right),
		\feq
		where $[d]:=\{1,\ldots,d\}.$
	\end{proposition}
	In fact, under the conditions of the proposition, all Lyapunov exponents of $\cala$ can be identified as distinct elements of the multi-set $\{\E \left( \log |A_1(i,i)|\right)\}_{i=1}^d$ (see Corollary~1 in \cite{hen} and item (iv) on p.~130 of \cite[Section~3.2]{arnold}).
	
	\par
	A conceptual explanation for Proposition~\ref{pkus} is provided by a lemma of Furstenberg and Kifer \cite{fifa} (quoted here as Lemma~\ref{blockt}), which reduces the computation of the top Lyapunov exponent for block-triangular products to that of the diagonal blocks; see, for example, \cite{fifa,kifer,blockas,blockus,tom}.  Extensions appear in several directions, including switched systems \cite{ots, solvable, new} and linear cocycles in bundles \cite[Lemma~3.6]{kifer}; see also \cite[Proposition~1]{blockas}, \cite[Theorem~1.1]{blockus} and \cite[Lemma~4.9]{tom}.  The starting point of this paper is that the same block-triangular viewpoint can be packaged into a practical toolkit that continues to be useful beyond the classical stationary-random setting.  In particular, the resulting bounds are explicit in terms of diagonal-block growth rates and simple combinatorial parameters of the zero pattern, which makes them amenable to computation in matrix-analytic applications.
	
	\par
	Our contributions are organized around four structured regimes.  Section~\ref{block} begins with deterministic control of upper triangular products: for general tempered sequences (not necessarily random), Proposition~\ref{thm100} gives a two-sided estimate for $\gamma_1$ in terms of diagonal growth rates $\alpha_i^\pm$, and Corollary~\ref{cora} recovers the familiar ``maximum of diagonal exponents'' identity when the relevant diagonal averages converge.  In the same section we develop a deterministic Furstenberg--Kifer-type reduction for block-triangular products under Lyapunov-regularity hypotheses on the diagonal blocks (Proposition~\ref{thm101}), providing a modular step that reduces $\gamma_1$ to diagonal-block data plus an explicit error term.
	
	Section~\ref{per_sect} illustrates how these reductions turn concrete perturbation models (including rank-one updates) into lower-dimensional or scalar cocycles with computable Lyapunov exponents, and it clarifies when a perturbation changes only the top exponent.  Finally, Section~\ref{shape} introduces shape graphs for sparse decompositions with disjoint supports and limited feedback, proves an energy--entropy bound (Theorem~\ref{th-shapes}) of the form $\gamma_1(\cala)\le \beta+\log k$ (or $\beta+\log k_*$ when $O_d\in\calv$), and provides a short ``how-to'' guide with worked examples in Section~\ref{sec:howto-shapes}, including a transfer-matrix/DAG model that makes the mechanism transparent.

	\section{Block-triangular matrices}
	\label{block}	
	We begin with the following inequality concerning the top Lyapunov exponent of a given sequence of upper triangular matrices (cf. \cite[Lemma~3.1.4]{barl}).
	\begin{proposition}
		\label{thm100}
		Let $\cala:=(A_n)_{n\in\nn}$ be a sequence of real, invertible upper triangular $d\times d$ matrices (that is $A_n(j,k)=0$ for $j>k$ and $A_n(j,j)\neq 0,$ $j,k\in [d]$) such that
		\beqn
		\label{small}
		\limsup_{n\to\infty}\frac{1}{n}\log \|A_n\|=0.
		\feqn
		Then,
		\beqn
		\label{D_ineq}
		\max_{j\in [d]} \alpha_j^+\leq \gamma_1
		\leq \max_{j\in [d]} \Big(\alpha_j^+  +\sum_{r=1}^{j-1} \big(\alpha_r^+-\alpha_r^-\big)\Big),
		\feqn
		where
		$
		\alpha_i^-:= \liminf_{n\to\infty} \frac{1}{n} \sum_{k=1}^n \log |A_k(i,i)|$ and $\alpha_i^+:= \limsup_{n\to\infty} \frac{1}{n} \sum_{k=1}^n \log |A_k(i,i)|.
		$
	\end{proposition}
	\begin{proof}
		The upper bound is an immediate consequence of Lemma~3.1.4 in \cite{barl} which shows that under the conditions of the proposition,
		\beq
		\limsup_{n\to\infty} \frac{1}{n} \log |X_n(i,j)|\leq \alpha_j^+ +\sum_{r=1}^{j-1} \big(\alpha_r^+-\alpha_r^- \big), \qquad \forall i,j\in [d].
		\feq
		The proof of the lower bound given in the following uses a connection between the growth rates of singular values and the Lyapunov exponents of $X_n$. For $A\in \calm_d,$ let $\sigma_1(A) \geq \cdots \geq \sigma_d(A)$ be singular values of $A,$ that is, eigenvalues of $(A^TA)^{1/2},$ where $A^T$ is the transpose of $A$. Since all norms on $\rr^d$ are equivalent, the top Lyapunov exponent $\gamma_1(\cala)$ does not depend on the particular choice of operator norm; in particular we may use the Euclidean norm when discussing singular values.
		Set
		\beqn
		\label{eli}
		\ell^{\sup}_i := \limsup_{n\to\infty} \frac{1}{n} \log \sigma_i(X_n) \qquad \mbox{\rm and} \qquad
		\ell^{\inf}_i := \liminf_{n\to\infty} \frac{1}{n} \log \sigma_i(X_n).
		\feqn
		For a general theory of these growth rates and their relation to the Lyapunov exponents in the case when $\cala=(A_n)_{n\in\nn}$ is non-random, see \cite{arnold, barr}. If $\cala$ is a random, stationary and ergodic sequence and \eqref{f_cond} holds, then (see, for instance, Theorem~3.4.1 in \cite{arnold}),
		\beq
		\ell^{\sup}_i= \ell^{\inf}_i=\gamma_i,\qquad \forall i\in [d],\quad \as
		\feq
		For a non random sequence $\cala,$ it has been shown in \cite{barr} that
		\beqn \label{u_bound1}
		\qquad \ell_1^{\sup}= \gamma_1 \qquad \mbox{\rm and}  \qquad \ell_i^{\sup} \leq \gamma_i \quad \mbox{for}~i>1,
		\feqn
		and, furthermore, the above inequalities may be strict. However, it turns out \cite{barl, barr} that if \eqref{small} holds (such sequences are referred to as tempered in \cite{barl,barr}) and the matrices in $\cala$ are invertible, then following properties are equivalent:
		\beqn
		\label{lr}
		\mbox{\rm (a)\, $(X_n^TX_n)^{\frac{1}{2n}}$ converges as $n\to\infty$ \quad and \quad (b)\, $\gamma_i=\ell_i^{\inf}=\ell_i^{\sup}$ for all $i\in [d].$}
		\feqn
		A sequence $\cala$ that satisfies these conditions is called Lyapunov regular, see \cite{barl} for a number of other equivalent conditions (cf. Remark~\ref{reb} below).
		\medskip
		
		\noindent\emph{Proof of lower bound:}
		Let $\lambda_1(X_n)$ be the largest by absolute value eigenvalue of $X_n.$ By Weyl's inequality \cite{horn}, $|\lambda_1(X_n)| \leq \sigma_1(X_n).$
		Since $X_n(j,j)=\prod_{i=1}^n A_i(j,j),$
		\beq
		\log |\lambda_1(X_n)|=\max_{j\in [d]} \sum_{i=1}^n  \log |A_i(j,j)|\leq \log \sigma_1(X_n),
		\feq
		which implies the lower bound in \eqref{D_ineq} by virtue of \eqref{eli} and the identity in \eqref{u_bound1}.
	\end{proof}

	Proposition~\ref{thm100} makes precise a common heuristic: for an upper triangular product, the diagonal controls the exponential growth, while off-diagonal terms can only create \emph{transient} amplification.
	The quantities $\alpha_i^\pm$ record the asymptotic growth rates of the diagonal cocycles along each coordinate.
	When the diagonal averages converge (so $\alpha_i^+=\alpha_i^-$), Corollary~\ref{cora} recovers the clean identity $\gamma_1=\max_i \alpha_i$.
	When they do not, the additional correction term $\sum_{r<j}(\alpha_r^+-\alpha_r^-)$ measures the maximal ``spread'' between limsup and liminf along earlier coordinates; this is exactly the amount of slack needed to control how often off-diagonal couplings can amplify the $j$th diagonal growth.
	In the stationary ergodic case the spread vanishes almost surely, so the proposition can be viewed as a deterministic envelope around the classical result.
	
	A deterministic analogue of Proposition~\ref{pkus} can now be stated as follows:
	
	\begin{corollary}
		\label{cora}
		Let the conditions of Proposition~\ref{thm100} hold for $\cala=(A_n)_{n\in\nn}.$ Suppose in addition that $\alpha_i^+=\alpha_i^-$ for all $i\in [d].$
		Then,
		$
		\gamma_1=\max_{j\in [d]} \lim_{n\to\infty} \frac{1}{n} \sum_{i=1}^n \log |A_i(j,j)|.
		$
	\end{corollary}
	We next state an extension of Proposition~\ref{pkus} to block triangular matrices
	in the form
	\beqn \label{block_def}
	A_n = \begin{bmatrix}
		B_n(1,1) & B_n(1,2) & \dots & B_n(1,m) \\
		O(2,1) & B_n(2,2) & \dots & B_n(2,m) \\
		\vdots & \vdots & \ddots & \vdots \\
		O(m,1) & O(m,2) & \dots & B_n(m,m)
	\end{bmatrix}
	,
	\feqn
	where $B_n(i,j)$ are $s_i\times s_j$ real matrices for a sequence of positive integers $(s_1,\ldots,s_m)$ with $s_1+\cdots+s_m=d$
	and $O(i,j)$ is a zero $s_i\times s_j$ matrix.
	\begin{lemma}
		\cite[Lemma~3.6]{fifa}
		\label{blockt}
		For given $m\in \nn$ and $(s_1,\ldots,s_m)\in\nn^m,$ let $(A_n)_{n\in\nn}$ be a random, stationary ergodic sequence of matrices, each having the shape introduced in \eqref{block_def}. Suppose in addition that for all $i\in [m]$, $B_1(i,i)$ is invertible with probability one, and for all $1\le i\le j\le m$,
		\beq
		\E\big(\log^+ \|B_1(i,j)\|\big)<\infty \qquad \mbox{and} \qquad \E\big(\log^+ \|B_1^{-1}(i,i)\|\big)<\infty.
		\feq
		For $i \in [m],$ let
		\beq
		\beta_i=\lim_{n\to\infty} \frac{1}{n}\log \big\|B_n(i,i)B_{n-1}(i,i)\cdots B_1(i,i)\big\|,\qquad \as
		\feq
		Then, $\gamma_1=\max_{i \in [m]} \beta_i.$
	\end{lemma}
	In the remainder of the paper, we occasionally write $\gamma_1(\cala)$ instead of $\gamma_1$ for the Lyapunov exponent of $\cala\in\calm_d^\nn.$  We next obtain the following deterministic version of Lemma~\ref{blockt}.
	\begin{proposition}
		\label{thm101}
		For given $m\in \nn$ and $(s_1,\ldots,s_m)\in\nn^m,$ let $\cala=(A_n)_{n\in\nn}$ be a tempered sequence of matrices (i.\,e., \eqref{small} holds), each having the shape introduced in \eqref{block_def}. For $i\in [m],$ let $\calb_i:=\big(B_n(i,i)\big)_{n\in \nn},$ and assume that all $B_n(i,i)$ are invertible and all $\calb_i$ are Lyapunov regular (i.\,e., an analogue of \eqref{lr} holds for $\calb_i$). Then, $\gamma_1(\cala)= \max_{j\in [m]} \gamma_1(\calb_j).$
	\end{proposition}	
	\begin{remark}
		\label{reb}
		\item [(a)] For an invertible cocycle $\cala=(A_n)_{n\in\nn}$, Barreira introduces the \emph{Grobman coefficient} (see, e.g., Section~2.2 in \cite{barr}) as a nonnegative quantity measuring the deviation from Lyapunov regularity; in particular, the cocycle is Lyapunov regular if and only if its Grobman coefficient is zero. Under the hypotheses of Proposition~\ref{thm101}, each diagonal block cocycle $\calb_j$ is Lyapunov regular, and hence its Grobman coefficient (in the sense of \cite{barr}) vanishes:
		\beqn
		\label{abs1}
		\mathrm{Grob}(\calb_j)=0,\qquad \forall~j\in [m],
		\feqn
		where $\mathrm{Grob}(\cdot)$ denotes Grobman coefficient.
		
		\item [(b)]
		Under the hypotheses of Lemma~\ref{blockt} (in particular the integrability of $B_1(i,i)^{-1}$), the smallest Lyapunov exponent of the diagonal block cocycle $\calb_i$ can be expressed via the inverse product (see, for instance, Theorem~3.3.10 in \cite{arnold}):
		\beqn
		\label{abs7}
		\lambda_{s_i}(\calb_i)=-\lim_{n\to\infty} \frac{1}{n}\log \big\|B_1^{-1}(i,i)\cdots B_n^{-1}(i,i)\big\|,\qquad \as,
		\feqn
		for all $i\in [m],$ where $\lambda_{s_i}(\calb_i)$ denotes the smallest Lyapunov exponent of $\calb_i$.
	\end{remark}
	
	\begin{proof}[Proof of Proposition~\ref{thm101}]
		Observe that $X_n$ has the same shape as the $A_n$ matrices. Write
		{\small 		\beq
			X_n = \begin{bmatrix}
				C_n(1,1) & C_n(1,2) & \dots & C_n(1,m) \\
				O & C_n(2,2) & \dots & C_n(2,m) \\
				\vdots & \vdots & \ddots & \vdots \\
				O & O & \dots & C_n(m,m)
			\end{bmatrix}
			\quad \, 
			H_n := \begin{bmatrix}
				C_n(1,1) & O & \dots & O \\
				O & C_n(2,2) & \dots & O \\
				\vdots & \vdots & \ddots & \vdots \\
				O & O & \dots & C_n(m,m)
			\end{bmatrix}
			,
			\feq
			where $C_n(i,j)$ are $s_i\times s_j$ matrices. Then,
			\beq
			\| X_n \|_F^2 = \| H_n \|_F^2 + \| X_n-H_n \|_F^2 + 2\inner{H_n}{X_n-H_n}_F,
			\feq
		}
		where $\inner{A}{B}_F=\sum_{i,j} A(i,j)B(i,j)$ is the Frobenius inner product of two matrices $A$ and $B,$ provided they have the same shape. Since for any pair $(i,j)\in [m]^2,$ either $H_n(i,j)$ or $(X_n-H_n)(i,j)$ is zero, $\inner{H_n}{X_n-H_n}_F = 0.$
		This implies that $\| X_n \|_F \geq \| H_n \|_F$. Since $H_n$ is block diagonal,
		\[
		\|H_n\|_F^2=\sum_{j=1}^m \|C_n(j,j)\|_F^2.
		\]
		Using the norm equivalence $\|M\|\le \|M\|_F\le \sqrt{d}\,\|M\|$ (valid for all $d\times d$ matrices), each block has the same exponential growth rate whether measured in $\|\cdot\|$ or $\|\cdot\|_F$. Consequently,
		\[
		\lim_{n\to\infty}\frac1n\log\|H_n\|_F=\max_{j\in[m]}\gamma_1(\calb_j)=:\beta,
		\]
		\noindent For each $i\in[m]$, set $\beta_i:=\gamma_1(\calb_i)$, so that $\beta_i\le \beta$.
		
		and taking $\frac1n\log$ in $\|X_n\|_F\ge \|H_n\|_F$ gives $\beta\le \gamma_1(\cala)$.
		We turn now  to the inverse inequality $\beta \geq \gamma_1(\cala).$ Let $d_0=0$ and $d_j=s_1+\cdots +s_j$ for $j\in[m].$ For $u\in \rr^d,$ let $u_j,$ $j=1,\ldots, m,$ be a vector in $\rr^{s_j}$ with
		\beq
		u_j(i)=u(d_{j-1}+i),\qquad i=1,\ldots, s_j.
		\feq
		Denote by $\|\cdot\|_1$ the $\ell^1$-norm in $\rr^d$ and corresponding matrix norm. Thus, $\|u\|_1=\sum_{i=1}^d |u(i)|$ and $\|A\|_1=\sup_{u\in\rr^d\backslash\{0\}} \frac{\|Au\|_1}{\|u\|_1}=\max_{j\in [d]} \sum_{i=1}^d |A(i,j)|$ for $A \in \calm_d.$ Then, for $u\in \rr^d,$
		\beq
		\|X_n u\|_1&=&\sum_{i=1}^m \Big\|\sum_{j=i}^m C_n (i,j)u_j\Big\|_1\leq \sum_{i=1}^m \sum_{j=i}^m\| C_n (i,j)\|_1\cdot \|u_j\|_1
		\\
		&\leq& \sum_{i=1}^m \sum_{j=i}^m\| C_n (i,j)\|_1\cdot \|u\|_1.
		\feq
		Therefore,
		\beqn
		\label{esta}
		\|X_n \|_1 \leq \sum_{i=1}^m \sum_{j=i}^m\| C_n (i,j)\|_1.
		\feqn
		In view of \eqref{esta}, it suffices to show that for all $i,j\in [m]$ such that $i\leq j,$
		\beqn
		\label{est3}
		\limsup_{n\to\infty} \frac{1}{n}\log \|C_n(i,j)\|_1\leq \beta ,\qquad \as
		\feqn
		To show \eqref{est3}, we adapt an inductive argument employed in the proof of Lemma~3.1.4 in \cite{barr} in order to obtain a version of this estimate for upper-triangular matrices. For $j>i,$ using a convention that $X_0$ is the $d\times d$ identity matrix, we have (cf. (3.10) and (3.11) in \cite{barr}):
		\beqn
		\nonumber
		C_n(i,j)&=&\sum_{t=i+1}^j B_n(i,t)C_{n-1}(t,j)+B_n(i,i)C_{n-1}(i,j)
		\\
		\nonumber
		&=&\sum_{t=i+1}^j B_n(i,t)C_{n-1}(t,j)+B_n(i,i)\sum_{t=i+1}^kB_{n-1}(i,t)C_{n-2}(t,j)
		\\
		\nonumber
		&& \qquad \qquad +B_n(i,i)B_{n-1}(i,i)C_{n-2}(i,j)= \cdots
		\\
		\label{est4}
		&=&
		\sum_{r=0}^{n-1} B_n(i,i)\cdots B_{n-r+1}(i,i)\sum_{t=i+1}^jB_{n-r}(i,t)C_{n-1-r}(t,j).
		\feqn
		Fix some $i^*,j^*\in [m]$ such that $i^*<j^*.$ The inequality in \eqref{est3} is trivially true for $i=j=j^*.$
		Suppose now that \eqref{est3} has been proved for all $(i,j)$ with $j=j^*$ and $i>i^*.$ To finish the proof of the theorem, we will next use a backward induction to show that it is valid for $(i,j)=(i^*,j^*).$ By \eqref{small} and the induction hypothesis, for all $\veps>0$ there exists a constant $D_\veps$ (which depends on $\cala$) such that for all $n\in\nn,$
		\beq
		\|B_n(i,j)\|_1\leq D_\veps e^{n \veps}, \qquad i,j\in [m],
		\feq
		and
		\beq
		\|C_n(i,j^*)\|_1\leq D_\veps e^{n (\beta +\veps)}, \qquad i^*<i\leq j^*.
		\feq
		Furthermore, it follows from \eqref{abs1} that for all $\veps>0$ there exists a constant $G_\veps$ (which depends on $\cala$) such that for any $i\in [m]$ and $k,n\in\nn$ with $k<n,$
		\beqn
		\nonumber
		\big\|B_n(i,i)\cdots B_k(i,i)\big\|_1&=&\big\|B_n(i,i)\cdots B_1(i,i)\cdot B_1^{-1}(i,i)\cdots B_{k-1}^{-1}(i,i)\big\|_1
		\\
		\label{gr}
		&\leq& \big\|B_n(i,i)\cdots B_1(i,i)\big\|_1\cdot \big \|B_1^{-1}(i,i)\cdots B_{k-1}^{-1}(i,i)\big\|_1
		\\
		\nonumber
		&\leq& \big(G_\veps e^{(\beta_i+\veps)n}\big)\cdot\big(G_\veps e^{(-\beta_i+\veps)(k-1)}\big)\leq G_\veps^2 e^{(n-k+1)\beta_i} e^{2n\veps}.
		\feqn
		Therefore, by virtue of \eqref{est4},
		\beq
		\|C_n(i^*,j^*)\|_1&\leq & \sum_{r=0}^{n-1} G_\veps^2 e^{\beta_i r}e^{2n\veps} \sum_{t=i^*+1}^{j^*}D_\veps^2 e^{n \veps} e^{(n-r+1)(\beta+\veps)}\\
		&\leq & nj^* G_\veps^2 D_\veps^2 e^{(4n+1) \veps} e^{(n+1)\beta},
		\feq
		which implies \eqref{est3} with $(i,j)=(i^*,j^*)$ since $\veps>0$ is arbitrary.
	\end{proof}
	
	\section{Application to matrix perturbations} 
	\label{per_sect}
	
	We next apply the results of Section~\ref{block} to perturbations of linear systems. These perturbations represent an important class of matrix transformations with wide-ranging applications in science and engineering.

	Our first result concerns the product of random rank-one perturbations of a specific form. This proposition can also be viewed somewhat as a generalization of Theorem~5.1 in \cite{rankone}, and contributes to the discussion around Remark~2.4 in the same paper.

	\par
	
	\begin{proposition} \label{pert_thm}
		Let $(u_n,\eta_n)_{n \in \nn}$ be a random, stationary, and ergodic sequence, where $u_n \in \rr^d$ and $\eta_n \in \rr$. Suppose $v \in \rr^d$ is a fixed vector such that
		\beqn \label{uv_cond}
		\E\bigl(\log |\eta_1|\bigr) \leq \E \bigl(\log |\eta_1 + v^T u_1|\bigr) < \infty.
		\feqn
		\begin{itemize}
			\item[(a)] Let $\cala := (A_n)_{n \in \nn}$ with $A_n := \eta_n I_d + u_n v^T$ satisfying \eqref{f_cond}. Then
			\beqn \label{An_perturbed_exponents}
			\gamma_1(\cala) = \E \bigl(\log |\eta_1 + v^T u_1| \bigr)
			\quad \text{and} \quad
			\gamma_r(\cala) = \E \bigl(\log |\eta_1|\bigr), \quad r > 1.
			\feqn
			
			\item[(b)] Suppose $A \in \calm_d$ is a non-singular matrix. Let $\cala := (A_n)_{n \in \nn}$ with $A_n := \eta_n A + u_n v^T$ satisfying condition \eqref{f_cond}, and
			\beqn \label{Auv_cond}
			P\bigl(A u_1 v^T = u_1 v^T A\bigr) = 1.
			\feqn
			Then,
			\beq
			- \log \rho(A^{-1}) \leq \gamma_1(\cala) - \E\bigl(\log|\eta_1 + v^T A^{-1} u_1|\bigr) \leq \log \rho(A),
			\feq
			where $\rho(A)$ denotes the spectral radius of the matrix $A$.
		\end{itemize}
	\end{proposition}
	
	We note that, under mild conditions on the sequence $(\eta_n, u_n)_{n \in \mathbb{N}}$, the inequality~\eqref{uv_cond} can be relaxed, since the result may be applied directly to the inverse cocycle $\cala^{-1}:=(A_n^{-1})_{n\in\nn}$, with part~(a) adjusted accordingly.
	
	\begin{remark}[Varying right vector: two extendable regimes]
		Consider $A_n=\eta_n I + u_n v_n^{\top}$.
		
		(i) If $v_n=c_n v$ for a fixed $v\neq 0$, then $A_n=\eta_n I + \tilde u_n v^{\top}$ with
		$\tilde u_n:=c_n u_n$. Hence Proposition~\ref{pert_thm} applies verbatim (under the
		corresponding integrability for $\tilde u_n$), yielding $d-1$ Lyapunov exponents equal to
		$\E\log|\eta_1|$ and the remaining exponent $\E\log|\eta_1+c_1 v^{\top}u_1|$.
		
		(ii) More generally, if $v_n\in V$ a.s. for a fixed deterministic subspace $V\subset \rr^d$
		of dimension $r$, then $H:=V^{\perp}$ is a common invariant subspace and
		$A_n|_H=\eta_n\,\mathrm{Id}$. Consequently $\E\log|\eta_1|$ is a Lyapunov exponent of
		multiplicity at least $d-r$. The remaining $r$ exponents are those of the induced cocycle
		on the quotient $\rr^d/H$; in general they are not explicit if $v_n$ rotates in $V$.
		Nevertheless, the determinant identity implies
		\[
		\sum_{i=1}^d \gamma_i=(d-1)\E\log|\eta_1|+\E\log|\eta_1+v_1^{\top}u_1|.
		\]
	\end{remark}

	In order to prove Proposition \ref{pert_thm}, we first present a preliminary result that explains how to extract the first few Lyapunov exponents of block upper triangular matrices by Lemma~\ref{blockt}. This extends Corollary~1 in \cite{hen} (cf.\ the paragraph following Proposition~\ref{pkus} in these notes).

	\begin{proposition} \label{block_all_exp}
		Fix $d, m \in \nn$, and let $\caln := (N_n)_{n \in \nn}$ be a random, stationary, and ergodic sequence of matrices
		\beq
		N_n := 
		\begin{bmatrix}
			A_n & * \\
			O & \eta_n I_m
		\end{bmatrix},
		\feq
		where $\eta_n \in \rr$ and $A_n \in \calm_d$. Assume that $\caln$ satisfies \eqref{f_cond} (i.e.\ $\E(\log^+\|N_1\|)<\infty$), that $A_1$ is almost surely invertible with $\E(\log^+\|A_1^{-1}\|)<\infty$, and that $\eta_1\neq 0$ almost surely with $\E(\log^+|\eta_1|^{-1})<\infty$. Set $\cala := (A_n)_{n \in \nn}$. Then
		\beq
		\gamma_1(\caln) = \max \bigl( \E(\log | \eta_1 |), \, \gamma_1(\cala) \bigr),
		\feq
		and for any $2 \leq r \leq \min(m,d)$,
		\beq
		&& \gamma_r(\caln) = \max_{0\leq \ell \leq r} \left(\sum_{s=1}^\ell \gamma_s(\cala) + (r-\ell) \E(\log | \eta_1|)  \right) \\
		&& \qquad \qquad - \max_{0\leq \ell \leq r-1} \left(\sum_{s=1}^\ell \gamma_s(\cala) + (r-\ell) \E(\log | \eta_1|)  \right).
		\feq
	\end{proposition}	
	\begin{proof}
		Recall that a minor of size~$r$ is associated to ordered sequences of indices $J$ and $L$ of length $r$, defined as the determinant of the submatrix formed by selecting the rows indexed by $J$ and the columns indexed by $L$.
		We write $[N]:=\{1,\dots,N\}$ and, for $0\le r\le N$, we let
		$[N]_r:=\{(i_1,\dots,i_r):1\le i_1<\cdots<i_r\le N\}$; we freely identify such increasing $r$--tuples with the underlying $r$--element subsets.
		For $B\subset[m]$ we write $B+d:=\{b+d: b\in B\}\subset\{d+1,\dots,d+m\}$, and for $J_2,L_2\subset\{d+1,\dots,d+m\}$ we interpret
		$I_m[J_2,L_2]:=I_m[J_2-d,\,L_2-d]$.
		For any matrix $A$, let $H_r(A)$ denote the $r$-th compound matrix of $A$, whose entries are all minors of size~$r$. Set $B_n^{(r)} := H_r(N_n)$, and let $\calb_r := (B_n^{(r)})_{n \in \nn}$ be the sequence of $r$-th compound matrices in $\caln$.
		
		Recall that by the discussion in p.~130 of \cite[Section~3.2]{arnold}, we write
		\beqn \label{comp_l_form}
		\gamma_1(\calb_r) = \sum_{i=1}^r \gamma_i(\caln)
		\quad \text{or equivalently} \quad
		\gamma_r(\caln) = 
		\begin{cases}
			\gamma_1(\calb_r) - \gamma_1(\calb_{r-1}) & r \geq 2, \\[4pt]
			\gamma_1(\calb_1) & r=1.
		\end{cases}
		\feqn
		
		Observe that $B_n^{(r)}$ is also a block upper triangular matrix, and any “diagonal” minor of $N_n$ can be written as the product of a minor of $\eta_n I_m$ and a minor of $A_n$. To formalize this, we define $H_0(A)=I_1$. For any two subsets $J,L \subset [m]$ of the same size, the $(J,L)$-th minor of $\eta_n I_m$ is zero if $J \neq L$, and equals $\eta_n^{|J|}$ otherwise. Let $J,L \in [m+d]_r$ be such that $J$ and $L$ can be written as $J_1 \cup J_2$ and $L_1 \cup L_2$, respectively, with $J_1, L_1 \in [d]_\alpha$ and $J_2, L_2 \in \{ B + d \mid B \in [m]_\beta \}$ for some $\alpha + \beta = r$. Then,
		\beq
		\det N_n[J,L] = \eta_n^{\beta} \det A_n[J_1, L_1] \det I_m[J_2-d,L_2-d].
		\feq
		
		Then, the observation is that for $2 \leq r \leq \min(m,d)$,
		\beq
		B_n^{(r)} = 
		\begin{bmatrix}
			H_r(\eta_n I_m) \otimes H_0(A_n) & * & * & \cdots & * \\
			O & H_{r-1}(\eta_n I_m) \otimes H_1(A_n) & * & \cdots & * \\
			O & O & \cdots & \cdots & * \\
			\vdots & \vdots & \ddots & \vdots & \vdots \\
			O & O & \cdots & O & H_0(\eta_n I_m) \otimes H_{r}(A_n)
		\end{bmatrix},
		\feq
		where $\otimes$ denotes the tensor product.
		Therefore, Lemma~\ref{blockt} yields
		\beq
		\gamma_1(\calb_r) 
		= \max_{0 \leq \ell \leq r} \gamma_1\bigl( \bigl(\eta_n^{r-\ell} H_\ell(A_n)\bigr)_{n \in \nn} \bigr)
		= \max_{0 \leq \ell \leq r} \Bigl( \sum_{s=1}^\ell \gamma_s(\cala) + (r-\ell) \E(\log | \eta_1 |) \Bigr).
		\feq
		Hence, by \eqref{comp_l_form}, the proof follows.
	\end{proof}

	Now, we provide the 
	
	\begin{proof}[Proof of Proposition~\ref{pert_thm}]
		We begin with the proof of (a). If $v=0$, then $A_n=\eta_n I_d$ and the conclusion is immediate. Assume $v\neq 0$.
		Choose $w\in\rr^d$ with $v^T w=1$ and extend a basis of $H:=\ker(v^T)$ to a basis of $\rr^d$ by adding $w$ as the last vector.
		Let $P$ be the corresponding change-of-basis matrix, so that in the new coordinates $v^T x$ is the last coordinate of $x$.
		A direct computation gives
		\[
		\tilde A_n:=P^{-1}A_nP=\eta_n I_d+\tilde u_n e_d^T,
		\qquad\text{where }\tilde u_n:=P^{-1}u_n\text{ and }e_d=(0,\dots,0,1)^T.
		\]
		In particular, each $\tilde A_n$ is upper triangular, with diagonal entries $\eta_n$ (with multiplicity $d-1$) and $\eta_n+e_d^T\tilde u_n=\eta_n+v^T u_n$.
		Since Lyapunov exponents are invariant under conjugation, the Lyapunov spectrum of $(A_n)$ equals that of $(\tilde A_n)$.
		Applying Proposition~\ref{thm100} (together with the multiplicative ergodic theorem under \eqref{f_cond}) to this triangular cocycle yields
		\[
		\gamma_1(\cala)=\E\bigl(\log|\eta_1+v^T u_1|\bigr)
		\quad\text{and}\quad
		\gamma_r(\cala)=\E\bigl(\log|\eta_1|\bigr),\ \ r>1,
		\]
		which is \eqref{An_perturbed_exponents}.
		
		\medskip
		Now we prove part (b). Define $\tilde \cala:=(\tilde A_n)_{n\in \nn}$ with $\tilde A_n := \eta_n I_d + A^{-1} u_n v^T$. Note that \(A_n = A \tilde A_n\). Repeatedly applying~\eqref{Auv_cond} gives
		\beqn \label{Xn_expansion0}
		X_n 
		= A(\eta_n I_d+ A^{-1} u_n v^T) \cdots A(\eta_1 I_d+ A^{-1} u_1 v^T) 
		= A^n Y_n,
		\feqn
		where $Y_n := \prod_{i=1}^n \tilde A_i$. Since the norm of a product is at most the product of norms, we get an upper bound:
		\beqn
		\gamma_1(\cala) 
		&\leq& \lim \frac{1}{n} \log \|A^n\| + \gamma_1(\tilde \cala) \notag \\
		&=& \log \rho(A) + \E\bigl(\log|\eta_1 + v^T A^{-1} u_1|\bigr), \label{upp_bnd1}
		\feqn
		where the last equality uses part~(a) applied to $\tilde\cala$ and Gelfand's formula for the spectral radius.
		
		Next, from \eqref{Xn_expansion0},
		\beq
		\|Y_n\| \leq \|X_n\| \, \|A^{-n}\|.
		\feq
		Taking logarithms in the previous inequality, dividing by $n$, and letting $n\to\infty$ (using part~(a) for the cocycle $(\tilde A_n)$ and Gelfand's formula for $\rho(A^{-1})$) yields the lower bound:
		\beqn
		\gamma_1(\cala) + \log \rho(A^{-1}) \geq \E\bigl(\log|\eta_1 + v^T A^{-1} u_1|\bigr). \label{lw_bnd1}
		\feqn
		Combining \eqref{lw_bnd1} and \eqref{upp_bnd1} concludes the proof of part~(b).
	\end{proof}
	
	Proposition~\ref{pert_thm} illustrates a general mechanism: a low-rank perturbation often preserves a large invariant subspace, and the cocycle induced on the corresponding quotient captures the remaining Lyapunov exponent(s).
	In our setting the hyperplane $H=\ker(v^T)$ is invariant and $A_n|_H=\eta_n I$, while the induced one-dimensional cocycle on $\rr^d/H$ is multiplication by $\eta_n+v^T u_n$ (as witnessed by the identity $v^T A_n=(\eta_n+v^T u_n)v^T$).
	Equivalently, after a fixed change of basis that sends $v^T x$ to the last coordinate, each $A_n$ becomes block upper triangular with diagonal blocks $\eta_n I_{d-1}$ and $\eta_n+v^T u_n$.
	
	This viewpoint is useful beyond rank one: when a cocycle preserves a filtration of subspaces, the Lyapunov spectrum is controlled by the induced cocycles on the successive quotients (cf.\ Lemma~\ref{blockt} and Proposition~\ref{block_all_exp}).

	\par 
	Finally, we extend this result to the following simple lemma by suitably adapting a standard $2 \times 2$ technique to our context (cf.\ for instance, \cite{lima, mann}).
	
	\begin{lemma} 
		\label{pthm1}
		For a given $m \in \nn$ and $V \in \calm_{d \times m}$, let $(U_n,\eta_n)_{n\in \nn}$ be a random, stationary, and ergodic sequence, where $U_n \in \calm_{d \times m}$ and $\eta_n \in \rr$. Define $\cala := (A_n)_{n\in \nn}$ and $\tilde \cala := (\tilde A_n)_{n\in \nn}$, with
		\[
		A_n := \eta_n I_d + U_n V^T
		\quad \text{and} \quad
		\tilde A_n := \eta_n I_m + V^T U_n.
		\]
		Then
		\beqn \label{sum_a_atilde}
		\sum_{i=1}^d \gamma_i(\cala) - \sum_{i=1}^m \gamma_i(\tilde \cala) = (d - m) \E(\log |\eta_1|).
		\feqn
		Moreover, if $\gamma_1(\cala) \geq \E(\log|\eta_1|)$, then $\gamma_1(\cala) = \gamma_1(\tilde \cala)$. Similarly, if for some $2 \leq r \leq \min(m,d)$ we have $\gamma_r(\cala) \geq \E(\log|\eta_1|)$, then $\gamma_r(\cala) = \gamma_r(\tilde \cala)$.
	\end{lemma}
	\begin{proof}
		Consider the following square block matrices:
		\[
		R :=
		\begin{bmatrix}
			I_d & O \\
			V^T & I_m
		\end{bmatrix},
		\quad
		R^{-1} =
		\begin{bmatrix}
			I_d & O \\
			- V^T & I_m
		\end{bmatrix},
		\qquad
		\text{and} \qquad
		N_n =
		\begin{bmatrix}
			A_n & U_n \\
			O & \eta_n I_m
		\end{bmatrix}.
		\]
		The proof of the lemma relies on the following identity:
		\beqn \label{cooliden_1}
		\tilde N_n := R N_n R^{-1} =
		\begin{bmatrix}
			\eta_n I_d & V \\
			O & \tilde A_n
		\end{bmatrix}.
		\feqn
		Recalling~\eqref{xn} and using~\eqref{cooliden_1}, we may write
		{\small
			\[
			\prod_{i=1}^n N_i =
			\begin{bmatrix}
				X_n & * \\
				O & \prod_{i=1}^n \eta_i I_m
			\end{bmatrix}
			= R^{-1}
			\begin{bmatrix}
				\prod_{i=1}^n \eta_i I_d & * \\
				0 & \tilde X_n
			\end{bmatrix}
			R 
			= R^{-1} \bigl(\prod_{i=1}^n \tilde N_i\bigr) R,
			\]
		}where \( * \) denotes unspecified blocks. Let $\caln:=(N_n)_{n\in \nn}$ and $\tilde \caln:=(\tilde N_n)_{n\in \nn}$. Observe that by the Oseledets theorem, $\gamma_i(\caln)$ and $\gamma_i(\tilde \caln)$s exist and
		\[
		\begin{aligned}
			\sum_{i=1}^d \gamma_i(\cala) + m \E(\log |\eta_1|) 
			&= \E\bigl(\log|\det(A_1)|\bigr) + m \E(\log |\eta_1|) \\
			&= \E\bigl(\log|\det(N_1)|\bigr) 
			= \sum_{i=1}^{d+m} \gamma_i(\caln) 
			= \sum_{i=1}^{d+m} \gamma_i(\tilde \caln) \\
			&= \E\bigl(\log|\det(\tilde N_1)|\bigr) 
			= \E\bigl(\log|\det(\tilde A_1)|\bigr) + d \E(\log |\eta_1|) \\
			&= \sum_{i=1}^m \gamma_i(\tilde \cala) + d \E(\log |\eta_1|).
		\end{aligned}
		\]
		Rearranging yields~\eqref{sum_a_atilde}.
		
		Next, $\gamma_1(\caln)$ and $\gamma_1(\tilde \caln)$ both exist and are equal. Hence, by Lemma~\ref{blockt},
		\beqn \label{max_gamma1}
		\max\bigl(\gamma_1(\cala), \E(\log|\eta_1|)\bigr) = \max\bigl(\gamma_1(\tilde \cala), \E(\log|\eta_1|)\bigr).
		\feqn
		If $\gamma_1(\cala) \geq \E(\log|\eta_1|)$, then~\eqref{max_gamma1} implies $\gamma_1(\cala)=\gamma_1(\tilde \cala)$.
		
		Similarly, by Proposition~\ref{block_all_exp}, for any $2 \leq r \leq \min(m,d)$, one can show inductively that
		\[
		\max_{0\leq \ell \leq r} \Bigl(\sum_{s=1}^\ell \gamma_s(\cala) + (r-\ell) \E(\log | \eta_1|) \Bigr)
		=
		\max_{0\leq \ell \leq r} \Bigl(\sum_{s=1}^\ell \gamma_s(\tilde \cala) + (r-\ell) \E(\log | \eta_1|) \Bigr).
		\]
		This yields the result under the assumption $\gamma_1(\cala) \geq \cdots \geq \gamma_r(\cala) \geq \E(\log | \eta_1|)$.
	\end{proof}
	
	The counterpart of part~(b) of Proposition~\ref{pert_thm} of this result is stated below.

	\begin{proposition}\label{pert_thm2}
		Fix $m,d\in\mathbb N$, a matrix $V\in\mathcal M_{d\times m}(\mathbb R)$ of rank $m$, and an invertible
		matrix $A\in\mathcal M_{d\times d}(\mathbb R)$.
		Let $(U_n,\eta_n)_{n\ge1}$ be a random stationary ergodic sequence with
		$U_n\in\mathcal M_{d\times m}(\mathbb R)$ and $\eta_n\in\mathbb R$, and define
		\[
		A_n:=\eta_n A + U_nV^T,\qquad 
		\widetilde A_n:=\eta_n I_m + V^T A^{-1}U_n .
		\]
		Assume
		\[
		\mathbb P(AU_1V^T = U_1V^T A)=1 .
		\]
		Then, writing $\cala=(A_n)$ and $\widetilde\cala=(\widetilde A_n)$,
		\[
		-\log\rho(A^{-1}) + \max\!\Big(\mathbb E\log|\eta_1|,\,\gamma_1(\widetilde\cala)\Big)
		\;\le\;
		\gamma_1(\cala)
		\;\le\;
		\log\rho(A) + \max\!\Big(\mathbb E\log|\eta_1|,\,\gamma_1(\widetilde\cala)\Big).
		\]
	\end{proposition}
	
	\begin{proof}
		Define
		\[
		\widehat A_n := \eta_n I_d + A^{-1}U_nV^T .
		\]
		By the commutation assumption, $A$ commutes with $\widehat A_n$ for every $n$, and hence
		\[
		X_n:=A_n\cdots A_1 = A^n \widehat X_n,
		\qquad\text{where } \widehat X_n:=\widehat A_n\cdots \widehat A_1 .
		\]
		Therefore, by submultiplicativity,
		\[
		\|X_n\|\le \|A^n\|\,\|\widehat X_n\|
		\quad\text{and}\quad
		\|\widehat X_n\|=\|A^{-n}X_n\|\le \|A^{-n}\|\,\|X_n\|.
		\]
		Taking $\frac1n\log$ and letting $n\to\infty$, and using Gelfand’s formula
		$\lim_{n\to\infty}\frac1n\log\|A^{\pm n}\|=\log\rho(A^{\pm1})$, gives
		\[
		-\log\rho(A^{-1})+\gamma_1(\widehat\cala)\le \gamma_1(\cala)\le \log\rho(A)+\gamma_1(\widehat\cala).
		\]
		
		It remains to identify $\gamma_1(\widehat\cala)$.
		Let $K:=\ker(V^T)$, which has dimension $d-m$ (since $\mathrm{rank}(V)=m$). For $x\in K$,
		$V^Tx=0$ and hence $\widehat A_n x=\eta_n x$, so vectors in $K$ grow at rate $\mathbb E\log|\eta_1|$.
		
		On the quotient space $\mathbb R^d/K$, the induced action is computed via the surjective map
		$q(x)=V^Tx$:
		\[
		q(\widehat A_n x)=V^T(\eta_n x + A^{-1}U_nV^Tx)
		=\bigl(\eta_n I_m + V^TA^{-1}U_n\bigr)\,q(x)
		=\widetilde A_n\,q(x).
		\]
		Thus the cocycle induced on the quotient is precisely $\widetilde\cala$.
		By the block/extension Furstenberg--Kifer lemma (Lemma~\ref{blockt}),
		\[
		\gamma_1(\widehat\cala)=\max\big(\mathbb E\log|\eta_1|,\gamma_1(\widetilde\cala)\big).
		\]
		Substituting into the earlier inequality yields the claimed bounds.
	\end{proof}

	\section{Shape graphs}
	\label{shape}	
	
	We conclude these notes by presenting a generalization of Lemma~\ref{blockt}, partially inspired by the proof of Theorem~A in~\cite{pinkus}. As a prelude to this discussion, we first provide an alternative proof of the upper bound in Lemma~\ref{blockt}, based on the observation that zeros in specific entries of \(A_n\) allow the computation of the Lyapunov exponents in terms of those associated with matrices of simpler structure. We note that the proof previously given for the theorem offers more precise insight into the growth behavior of the blocks and, in particular, can be adapted to establish Proposition~\ref{thm100}. 
	
	For the purpose of this section, let \(\mathcal{A}:=(A_n)_{n\in \mathbb{N}}\) be a random, stationary, and ergodic sequence of matrices as in the statement of Lemma~\ref{blockt}, and write each matrix \(A_n\) as
	\[
	A_n := A_{n,1} + A_{n,2},
	\]
	where \(A_{n,1}\) denotes the block diagonal part and \(A_{n,2}\) the remaining strictly upper triangular component of \(A_n\). \par\noindent{\bf Norm convention.} Throughout Section~\ref{shape} we take $\|\cdot\|$ to be the operator norm induced by the vector $\ell^1$-norm (maximum absolute column sum); we write $\|\cdot\|_1$ when clarity is helpful. By equivalence of norms on the finite-dimensional space $\mathcal M_d$, this choice does not change the Lyapunov exponents and \eqref{f_cond} holds for $\|\cdot\|$ if and only if it holds for $\|\cdot\|_1$. Then
	\[
	X_n = A_{n,1}\cdots A_{1,1}
	+ \sum_{j=1}^n \sum_{1\le i_1 < \cdots < i_j\le n}
	A_{n,1}\cdots A_{i_j+1,1}\, A_{i_j,2}\, A_{i_j-1,1}\cdots A_{i_{j-1}+1,1}\, A_{i_{j-1},2}\,\cdots.
	\]
	It is straightforward to verify that multiplying a block diagonal matrix by a strictly block upper triangular matrix (with the same block sizes) produces a strictly block upper triangular matrix. Moreover, multiplying two strictly upper triangular matrices yields another strictly upper triangular matrix with fewer nonzero superdiagonals—a property reminiscent of the nilpotency of strictly upper triangular matrices. Consequently, all products with \(j > d\) vanish, and therefore
	\[
	X_n = A_{n,1}\cdots A_{1,1}
	+ \sum_{j=1}^d \sum_{1 \leq i_1 < \cdots < i_j \leq n}
	A_{n,1}\cdots A_{i_j+1,1}\, A_{i_j,2}\, A_{i_j-1,1}\cdots A_{i_{j-1}+1,1}\, A_{i_{j-1},2}\,\cdots.
	\]
	
	Fix $\varepsilon>0$. We will use the following standard Borel--Cantelli estimate (applied to $Y_n=\log^+\|A_{n,2}\|_1$).
	
	\begin{lemma}\label{lem:bc_log}
		Let $(Y_n)_{n\ge 1}$ be a stationary sequence of nonnegative random variables with $\E(Y_1)<\infty$. Then for every $\varepsilon>0$,
		\[
		\sum_{n\ge 1} \pp(Y_n>\varepsilon n)<\infty,
		\qquad\text{and hence}\qquad 
		\pp(Y_n>\varepsilon n\ \text{\rm i.o.})=0.
		\]
	\end{lemma}
	\begin{proof}
		By stationarity, $\pp(Y_n>\varepsilon n)=\pp(Y_1>\varepsilon n)$. Using the tail-integral identity
		$$\E(Y_1)=\int_0^\infty \pp(Y_1>t)\,dt$$ and partitioning the integral into intervals of length $\varepsilon$ gives
		\[
		\sum_{n\ge 1}\pp(Y_1>\varepsilon n)
		\le \frac{1}{\varepsilon}\int_0^\infty \pp(Y_1>t)\,dt
		= \frac{1}{\varepsilon}\E(Y_1)<\infty.
		\]
		The Borel--Cantelli lemma yields the claim.
	\end{proof}
	
	Apply Lemma~\ref{lem:bc_log} with $Y_n:=\log^+\|A_{n,2}\|_1$. Since $A_{1,2}$ is obtained from $A_1$ by zeroing some entries, the induced $\ell^1$ operator norm is monotone, so $\|A_{1,2}\|_1\le \|A_1\|_1$ and hence $\E\log^+\|A_{1,2}\|_1<\infty$ by \eqref{f_cond} (equivalently for $\|\cdot\|_1$). Therefore,
	\[
	\pp\!\left(\log^+\|A_{n,2}\|_1>\varepsilon n\ \text{\rm i.o.}\right)=0.
	\]
	In particular, almost surely there exists a (random) $N_\varepsilon<\infty$ such that $\|A_{n,2}\|_1\le e^{\varepsilon n}$ for all $n\ge N_\varepsilon$. Setting
	\[
	C_\varepsilon:=\max\Bigl\{1,\ \max_{1\le n\le N_\varepsilon}\|A_{n,2}\|_1 e^{-\varepsilon n}\Bigr\},
	\]
	we obtain the uniform bound
	\[
	\|A_{n,2}\|_1\le C_\varepsilon e^{\varepsilon n}\qquad\text{for all }n\ge 1.
	\]
	
	Consequently, using the expansion of $X_n$ above and the fact that any nonzero monomial contains at most $d-1$ factors from the strictly upper-triangular part $(A_{k,2})$, we obtain that for all sufficiently large $n$, almost surely,
	\begin{equation}\label{X_n_uppebound1}
		\|X_n\|_1
		\le C_\varepsilon^{\,d-1}\, n^{d}\, e^{(d-1)\varepsilon n}\,
		\max_{1\le j<d}\ \max_{1\le i_1<\cdots<i_j\le n}\ 
		\prod_{r=0}^{j}\big\|A_{i_r+1,1}\cdots A_{i_{r+1}-1,1}\big\|_1,
	\end{equation}
	where by convention $i_0:=0$ and $i_{j+1}:=n+1$, and an empty product is interpreted as the identity matrix. (We used the crude bound $1+\sum_{j=1}^{d-1}\binom{n}{j}\le n^d$ on the number of summands in the expansion.)
	
	Let \(\mathcal{A}_1 := (A_{n,1})_{n \in \mathbb{N}}\). Note that $A_{n,1}$ is diagonal (hence block diagonal), so for any $1\le i\le j$,
	\[
	\|A_{j,1}\cdots A_{i,1}\|_1=\max_{p\in[d]}\ \prod_{t=i}^{j}\big|(A_{t,1})_{pp}\big|.
	\]
	For each $p\in[d]$, set
	\[
	\beta_p:=\lim_{n\to\infty}\frac1n\log\prod_{t=1}^{n}\big|(A_{t,1})_{pp}\big|\qquad\text{a.s.},
	\qquad\text{so that}\qquad
	\gamma_1(\mathcal{A}_1)=\max_{p\in[d]}\beta_p.
	\]
	Fix $\varepsilon>0$. Since $\frac1n\sum_{t=1}^{n}\log |(A_{t,1})_{pp}|\to \beta_p$ almost surely, the random variables
	\[
	G_{\varepsilon,p}^+:=\sup_{n\ge1}\exp\!\Big(\sum_{t=1}^{n}\log |(A_{t,1})_{pp}|-(\beta_p+\varepsilon)n\Big),
	\qquad
	G_{\varepsilon,p}^-:=\sup_{n\ge1}\exp\!\Big(-\sum_{t=1}^{n}\log |(A_{t,1})_{pp}|-(-\beta_p+\varepsilon)n\Big)
	\]
	are almost surely finite, and therefore with $G_{\varepsilon,p}:=\max(G_{\varepsilon,p}^+,G_{\varepsilon,p}^-)$ we have for all $n\ge1$,
	\[
	\prod_{t=1}^{n}\big|(A_{t,1})_{pp}\big|\le G_{\varepsilon,p}\,e^{(\beta_p+\varepsilon)n},
	\qquad
	\prod_{t=1}^{n}\big|(A_{t,1})_{pp}\big|^{-1}\le G_{\varepsilon,p}\,e^{(-\beta_p+\varepsilon)n}.
	\]
	Consequently, for all $1\le i\le j$,
	\beq
	&& \prod_{t=i}^{j}\big|(A_{t,1})_{pp}\big|
	=\Big(\prod_{t=1}^{j}\big|(A_{t,1})_{pp}\big|\Big)\Big(\prod_{t=1}^{i-1}\big|(A_{t,1})_{pp}\big|\Big)^{-1} \\
	&& \qquad \le G_{\varepsilon,p}^2\, e^{\beta_p(j-i+1)}\,e^{\varepsilon(j+i-1)}
	\le G_{\varepsilon,p}^2\, e^{\beta_p(j-i+1)}\,e^{2\varepsilon j}.
	\feq
	Taking the maximum over $p\in[d]$ and setting $G_\varepsilon:=\max_{p\in[d]}G_{\varepsilon,p}$ yields the uniform estimate
	\begin{equation}\label{eq:diag-seg-bound}
		\|A_{j,1}\cdots A_{i,1}\|_1\le G_\varepsilon^2\, e^{\gamma_1(\mathcal{A}_1)(j-i+1)}\,e^{2\varepsilon j}
		\qquad\text{for all }1\le i\le j.
	\end{equation}
	
	Applying \eqref{eq:diag-seg-bound} to each factor in \eqref{X_n_uppebound1} and using that each endpoint is at most $n$ and that $j\le d-1$ gives
	\[
	\|X_n\|_1
	\le \widetilde C_\varepsilon\, n^{d}\, \exp\!\big((\gamma_1(\mathcal{A}_1)+C\,\varepsilon)n\big),
	\]
	for some deterministic constant $C=C(d)$ and an almost surely finite (random) constant $\widetilde C_\varepsilon$.
	Since $\varepsilon>0$ is arbitrary, we conclude $\gamma_1(\mathcal{A}) \le \gamma_1(\mathcal{A}_1)$.
	Finally, for a diagonal (more generally, block-diagonal) cocycle, the top Lyapunov exponent is the maximum of the top Lyapunov exponents of its diagonal blocks, completing the proof.

	We now introduce several definitions to abstract and extend the preceding argument to a broader class of matrices beyond the block triangular case.
	
	Let \(\wedge\) denote the entrywise “AND” operation. Fix \(d, k \in \mathbb{N}\) with \(k \leq d^2\).  
	A \emph{shape set of order~\(d\)}, denoted \(\mathcal{L}\), is a collection of nonzero, zero–one matrices of size \(d \times d\),
	\[
	\mathcal{L} := \{ L_j \mid j \in [k] \},
	\]
	such that \(L_i \wedge L_j = O_d\) for all distinct \(i, j \in [k]\).  
	For any matrix \(A \in \calm_d\), we define its \emph{shape}, denoted \(\shape{A}\), as the \(d \times d\) zero–one matrix whose \((i,j)\)-th entry is one if \(A(i,j) \neq 0\), and zero otherwise.
	
	We next introduce the notion of an \emph{\(\mathcal{L}\)-shape graph} \((\mathcal{V}, \mathcal{T})_{\mathcal{L}}\), defined by its set of vertices~\(\mathcal{V}\) and transition map \(\mathcal{T} : \mathcal{V} \times \mathcal{L} \to \mathcal{V}\).  
	Given the shape set~\(\mathcal{L}\), define the set of vertices as
	\begin{equation} \label{V_def}
		\mathcal{V} := \Bigl\{ \shape{L_n\cdots L_1} \,\big|\, n \in \mathbb{N},\ (L_1, \ldots, L_n) \in \mathcal{L}^n \Bigr\}.
	\end{equation}
	Each vertex can be regarded as a zero–one matrix when necessary.  
	Two vertices \(v_1, v_2 \in \mathcal{V}\) are connected, written \(v_1 \to v_2\), if and only if there exists some \(L \in \mathcal{L}\) such that
	\[
	v_2 = \mathcal{T}(v_1, L) := \shape{L v_1}.
	\]
	
	Finally, given any sequence of matrices \(\mathcal{A} := (A_n)_{n \geq 0}\), we use the shape set~\(\mathcal{L}\) to decompose each \(A_n\) into a sum of \(k\) “non-overlapping” matrices \(\{A_{n,1}, \ldots, A_{n,k}\}\) such that
	\begin{equation} \label{A_ni_def}
		A_n = \sum_{j=1}^k A_{n,j}, 
		\quad \text{where} \quad \shape{A_{n,j}} \vee L_j = L_j.
	\end{equation}
	
	To clarify these definitions, we conclude this section with several examples, including one illustrating a shape graph appropriate for triangular matrices.

	\begin{figure} 	
		\centering
		\begin{tikzpicture}[shorten >=1pt,node distance=1.75cm,on grid,auto]
			\node[state] (q_0)   {$v_1$};
			\node[state] (q_1) [right=of q_0] {$v_2$};
			\node[state] (q_2) [right=of q_1] {$v_3$};
			\node[state](q_3) [right=of q_2] {$v_*$};
			\path[->]
			(q_0) edge  node {$L_2$} (q_1)
			edge [loop above] node {$L_1$} ()
			(q_1) edge  node  {$L_2$} (q_2)
			edge [loop above] node {$L_1$} ()
			(q_2) edge  node {$L_2$} (q_3)
			edge [loop above] node {$L_1$} ()
			(q_3) edge  node {$ $} (q_3)
			edge [loop above] node {$L_1, L_2$} ();
		\end{tikzpicture}
		\caption{The shape graph for \eqref{exm1-d1} and \eqref{exm1-v1}.}
		\label{exm1-fig}
	\end{figure}
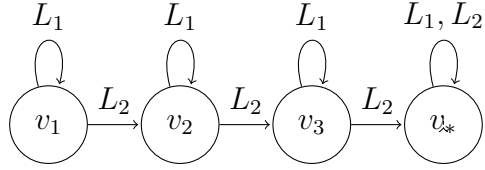
	
	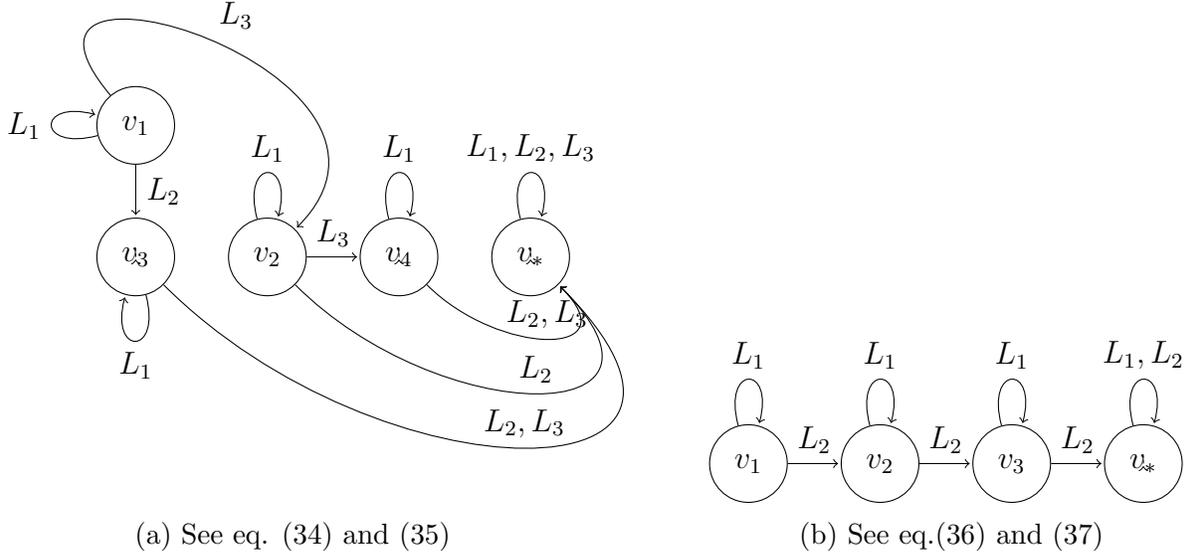
\begin{figure}[!t]
		\begin{subfigure}{.525\textwidth} \centering
			\begin{tikzpicture}[shorten >=1pt,node distance=1.75cm,on grid,auto]
				\node[state] (q_0)   {$v_1$};
				\node[state] (q_1) [below=of q_0] {$v_3$};
				\node[state] (q_2) [right=of q_1] {$v_2$};
				\node[state] (q_3) [right=of q_2] {$v_4$};
				\node[state](q_4) [right=of q_3] {$v_*$};
				\path[->]
				(q_0) edge  node {$L_2$} (q_1)
				edge [loop left] node {$L_1$} ()
				(q_1) edge  node  {$ $} (q_1)
				edge [loop below] node {$L_1$} ()
				(q_2) edge  node {$L_3$} (q_3)
				edge [loop above] node {$L_1$} ()
				(q_3) edge  node {$ $} (q_3)
				edge [loop above] node {$L_1$} ()
				(q_4) edge  node {$ $} (q_4)
				edge [loop above] node {$L_1, L_2, L_3$} ();
				\draw[->]
				(q_0) to [out=130,in=45,looseness=2.5] node {$L_3$} (q_2);
				\draw[->]
				(q_1) to [out=315, in=315,looseness=2] node {$L_2,L_3$} (q_4);
				\draw[->]
				(q_2) to [out=315, in=315,looseness=2] node {$L_2$} (q_4);
				\draw[->]
				(q_3) to [out=315, in=315,looseness=2] node {$L_2,L_3$} (q_4);
			\end{tikzpicture}
			\caption{See eq. \eqref{exm2_d1} and \eqref{exm2_v1}} 
		\end{subfigure}
		\begin{subfigure}{.525\textwidth} \centering
			\begin{tikzpicture}[shorten >=1pt,node distance=1.75cm,on grid,auto]
				\node[state] (q_0)   {$v_1$};
				\node[state] (q_1) [right=of q_0] {$v_2$};
				\node[state] (q_2) [right=of q_1] {$v_3$};
				\node[state] (q_3) [right=of q_2] {$v_*$};
				\path[->]
				(q_0) edge  node {$L_2$} (q_1)
				edge [loop above] node {$L_1$} ()
				(q_1) edge  node  {$L_2$} (q_2)
				edge [loop above] node {$L_1$} ()
				(q_2) edge  node {$L_2$} (q_3)
				edge [loop above] node {$L_1$} ()
				(q_3) edge  node {$ $} (q_3)
				edge [loop above] node {$L_1, L_2$} ();
			\end{tikzpicture}
			\caption{See eq.\eqref{exm2_d2} and \eqref{exm2_v2}} 
		\end{subfigure}	
		\caption{Two different graphs for the matrices described by \eqref{two_exm}}
		\label{exm7-fig}	
	\end{figure}	
	
	\begin{example}
		Let $\cala$ be a sequence of $3 \times 3$ upper-triangular matrices. Write each $A_i$ as the sum of its diagonal part and its strictly upper triangular part. Then, an appropriate shape set is
		\beqn \label{exm1-d1}
		\call := \left\{ 
		L_1 := I_3,\ 
		L_2 := 
		\begin{bmatrix}
			0 & 1 & 1 \\
			0 & 0 & 1 \\
			0 & 0 & 0
		\end{bmatrix} 
		\right\}.
		\feqn
		By \eqref{V_def}, the set of vertices of the $\call$-shape graph is then
		\beqn \label{exm1-v1}
		\calv := \left\{ 
		v_1 := I_3,\ 
		v_2 := 
		\begin{bmatrix}
			0 & 1 & 1 \\
			0 & 0 & 1 \\
			0 & 0 & 0
		\end{bmatrix},\ 
		v_3 := 
		\begin{bmatrix}
			0 & 0 & 1 \\
			0 & 0 & 0 \\
			0 & 0 & 0
		\end{bmatrix},\ 
		v_* := O_3 
		\right\}.
		\feqn
		The corresponding shape graph is shown in Figure~\ref{exm1-fig}.
	\end{example}
	
	Our next example concerns the case where the matrices involved are not both upper triangular.
	
	\begin{example}
		Let $\cala$ be a sequence of matrices sampled from the set
		\beqn \label{two_exm}
		\left\{ 
		\begin{bmatrix}
			2  & 0 & 0 & 0 \\
			0 & 1 & -1 & -1 \\
			0 & 0 & -1 & 0 \\
			0 & 0 & 0 & -2
		\end{bmatrix},\ 
		\begin{bmatrix}
			1  & 0 & 0 & 0 \\
			0 & 2 & 0 & 0 \\
			0 & 0 & 3 & 0 \\
			5 & 0 & 0 & 4
		\end{bmatrix} 
		\right\}.
		\feqn
		In this case, an appropriate choice for the shape set $\call$ is
		\beqn \label{exm2_d1}
		\call = \left\{
		L_1:=I_4,\ 
		L_2:= 
		\begin{bmatrix}
			0  & 0 & 0 & 0 \\
			0 & 0 & 0 & 0 \\
			0 & 0 & 0 & 0 \\
			1 & 0 & 0 & 0
		\end{bmatrix},\ 
		L_3:= 
		\begin{bmatrix}
			0  & 0 & 0 & 0 \\
			0 & 0 & 1 & 1 \\
			0 & 0 & 0 & 0 \\
			0 & 0 & 0 & 0
		\end{bmatrix}
		\right\}.
		\feqn
		It is straightforward to verify that with this choice of $\call$, the set of vertices is
		\beqn \label{exm2_v1}
		\calv := \left\{ 
		v_1:=L_1,\ 
		v_2:=L_2,\ 
		v_3:=L_3,\ 
		v_4:= 
		\begin{bmatrix}
			0  & 0 & 0 & 0 \\
			1 & 0 & 0 & 0 \\
			0 & 0 & 0 & 0 \\
			0 & 0 & 0 & 0
		\end{bmatrix},\ 
		v_*:=O_4 
		\right\}.
		\feqn
		The corresponding shape graph is depicted in part~(a) of Fig.~\ref{exm7-fig}. We note that the shape set $\call$ is not unique, and therefore the resulting graph is also not unique. For instance, another possible choice for $\call$ in this example is
		\beqn \label{exm2_d2}
		\call = \left\{  
		L_1 := I_4,\ 
		L_2 := 
		\begin{bmatrix}
			0 & 0 & 0 & 0 \\
			0 & 0 & 1 & 1 \\
			0 & 0 & 0 & 0 \\
			1 & 0 & 0 & 0
		\end{bmatrix}
		\right\} .
		\feqn
		In this case, the set of vertices becomes
		\beqn \label{exm2_v2}
		\calv := \left\{ 
		v_1 := L_1,\ 
		v_2 := L_2,\ 
		v_3 := 
		\begin{bmatrix}
			0 & 0 & 0 & 0 \\
			1 & 0 & 0 & 0 \\
			0 & 0 & 0 & 0 \\
			0 & 0 & 0 & 0
		\end{bmatrix},\ 
		v_* := O_4 
		\right\}.
		\feqn
		The corresponding shape graph is shown in part~(b) of Fig.~\ref{exm7-fig}.
	\end{example}

	With the notion of the shape graph established, our strategy is to exploit its structural properties to analyze products of matrices in $\cala$, thereby obtaining meaningful bounds on the Lyapunov exponents of $\cala$. The main result of this section is the following theorem.
	
	\begin{theorem}
		\label{th-shapes}
		Let $\call$ be a shape set of order $d$ with $k$ elements, and let $(\calv, \calt)_{\call}$ denote its corresponding shape graph. Assume that the shape graph contains only loops of length one (self-loops), and that each nonzero vertex has at most one self-loop. Let $\cala = (A_n)_{n \in \nn}$ be a random, stationary, and ergodic sequence of $d \times d$ matrices, each admitting a decomposition as in~\eqref{A_ni_def} with respect to the shape set $\call$. Assume further that:
		\begin{itemize}
			\item[(i)] For every $1 \leq s \leq k$,
			\[
			\E\big(\log^+ \|A_{1,s}\|\big) < \infty.
			\]
			\item[(ii)] Let $\calh \subset \calv \setminus \{ O_d \}$ be the set of vertices with a self-loop, and let $\call_{\calh}$ be the subset of labels corresponding to their self-loops. For every $s \in \call_\calh$, the matrix $A_{1,s}$ is almost surely invertible, and
			\[
			\E\big(\log^+ \|A_{1,s}^{-1}\|\big) < \infty.
			\]
		\end{itemize}
		Define
		\[
		\beta_s := \lim_{n \to \infty} \frac{1}{n} \log \left\| A_{n,s} A_{n-1,s} \cdots A_{1,s} \right\|, \quad \text{a.s.},
		\]
		and set
		\[
		\beta := \max_{s \in \call_\calh} \beta_s.
		\]
		
		Then:
		\begin{itemize}
			\item[(a)] The top Lyapunov exponent satisfies
			\[
			\gamma_1(\cala) \leq \beta + \log k.
			\]
			Moreover, if $O_d \in \calv$, define
			\[
			k_* := \max_{v \in \calv \setminus \{O_d\}} \#\bigl\{\, s \in [k] : \calt(v,L_s)\neq O_d \,\bigr\},
			\]
			so that $k_* \le k$ and $k_*$ is the maximal number of labels that keep a nonzero vertex from transitioning to $O_d$ in one step. Then
			\[
			\gamma_1(\cala) \leq \beta + \log k_*.
			\]
			(In particular, if every nonzero vertex has at least one outgoing label that transitions to $O_d$, then $k_* \le k-1$.)
			\item[(b)] Let $\calw \subset \calh\setminus \{O_d\}$ be the set of self-loop vertices $w$ such that every directed path in the shape graph that ends at $w$ avoids loop vertices other than $w$ (equivalently, it does not pass through any vertex in $\calh\setminus\{w\}$). Assume moreover that for each $w\in\calw$, if $s$ denotes the label of its self-loop, then there exists $r=r(w)\ge 1$ such that $\shape{L_s^n}=w$ for all $n\ge r$. If, for every $w \in \calw$ and every $v \in \calv\setminus\{w\}$, we have
			\[
			v \wedge w = O_d,
			\]
			and assume in addition that all matrices $A_{n,s}$ are entrywise nonnegative almost surely (so that monomials with the same shape do not cancel entrywise). Then
			\[
			\max_{s \in \call_\calw} \beta_s \leq \gamma_1(\cala),
			\]
			where $\call_\calw$ denotes the set of labels corresponding to the self-loops of vertices in $\calw$.
		\end{itemize}
	\end{theorem}

	\subsection{Theorem~\ref{th-shapes} in practice}
	\label{sec:howto-shapes}
	
	Theorem~\ref{th-shapes} is intended as a ``plug-in'' estimate for cocycles whose random matrices admit a sparse decomposition with \emph{non-overlapping supports}.  In many situations, the verification reduces to a finite, combinatorial check on the associated shape graph.
	We record a short checklist.
	
	\begin{enumerate}
		\item \textit{Choose a shape set.}  Pick a shape set $\call=\{L_1,\ldots,L_k\}$ (disjoint supports: $L_i\wedge L_j=O_d$ for $i\neq j$, where $O_d$ is the $d\times d$ zero matrix) that matches the sparsity types present in $A_n$.
		
		\item \textit{Build the shape graph.}  Compute (or reason about) the vertex set $\calv$ in \eqref{V_def} and the transition map $\calt(v,L)=\shape{vL}$.  In practice, $\calv$ is generated by iterating the transitions starting from the labels.
		
		\item \textit{Check the acyclicity condition.}  Verify that the directed graph contains no directed cycles other than self-loops, and that each nonzero vertex has at most one self-loop label.  (In many common constructions, $L_1=I_d$ is the unique self-loop label.)
		
		\item \textit{Identify the loop labels and compute $\beta$.}  Let $\calh$ be the loop vertices and $\call_{\calh}$ their self-loop labels.  Then $\beta=\max_{s\in\call_{\calh}}\beta_s$, where $\beta_s$ is the top Lyapunov exponent of the subcocycle $(A_{n,s})_{n\ge1}$.
		
		\item \textit{Apply the bound.}  Theorem~\ref{th-shapes}(a) gives $\gamma_1(\cala)\le \beta+\log k$, and if $O_d\in\calv$ the sharper $\gamma_1(\cala)\le \beta+\log k_*$.
	\end{enumerate}
	
	We now give two worked illustrations, the first using the non-triangular example introduced above, and the second describing a general ``DAG-supported'' family where the $\log k$ term has a transparent combinatorial meaning.
	
	\begin{example}[A worked bound for the non-triangular example \eqref{two_exm}]
		Assume that $(A_n)_{n\ge1}$ is i.i.d.\ and takes the first matrix in \eqref{two_exm} with probability $p\in(0,1)$ and the second with probability $1-p$.  Use the shape set $\call=\{L_1,L_2,L_3\}$ from \eqref{exm2_d1}, so that $k=3$, and note from Fig.~\ref{exm7-fig}(a) that $O_4\in\calv$ and the shape graph has no directed cycles other than self-loops.
		
		The only self-loop label is $L_1=I_4$, hence $\call_{\calh}=\{1\}$ and $\beta=\beta_1$.  Writing $D_n:=A_{n,1}$ for the diagonal component in the decomposition \eqref{A_ni_def}, we have that $D_n$ is diagonal with entries distributed as
		\[
		(D_n(1,1),D_n(2,2),D_n(3,3),D_n(4,4))=
		\begin{cases}
			(2,1,-1,-2), & \text{with prob. }p,\\
			(1,2,3,4), & \text{with prob. }1-p.
		\end{cases}
		\]
		For diagonal cocycles, $\beta_1$ is explicit:
		\[
		\beta_1=\max_{1\le i\le4}\ \E\big(\log|D_1(i,i)|\big)
		=
		\max\Big\{p\log 2,\ (1-p)\log 2,\ (1-p)\log 3,\ p\log 2+(1-p)\log 4\Big\}.
		\]
		Therefore Theorem~\ref{th-shapes}(a) yields the concrete estimate
		\[
		\gamma_1(\cala)\le \beta_1+\log k_*=\beta_1+\log 2 \qquad(\text{here }k_*=2).
		\]
		This bound is nontrivial in regimes where the diagonal growth $\beta_1$ is modest while the off-diagonal entries induce large transient amplification; the shape-graph estimate isolates the genuine exponential growth (captured by $\beta_1$) from the purely combinatorial proliferation of monomials (captured by $\log 2$).
	\end{example}
	
	\begin{example}[DAG-supported sparsity patterns]
		
		One convenient way to view this family is as a \emph{random transfer-matrix} (or random kinetic-network) model on a directed acyclic network.
		Let $x_n\in\mathbb{R}^d$ be a vector of nonnegative ``masses'' (or partial partition functions) and evolve it by
		\[
		x_n = A_n x_{n-1},\qquad n\ge1.
		\]
		Then $\gamma_1(\cala)$ is the almost-sure exponential growth rate of $\|x_n\|$ and may be interpreted as a \emph{quenched free-energy density}.
		A concrete parameterization is to take (for an ``inverse temperature'' $\theta>0$)
		\[
		D_n=\diag{e^{-\theta E_{n,1}},\ldots,e^{-\theta E_{n,d}}},
		\qquad 
		B_{n,e}=e^{-\theta U_{n,e}}\,E_{ij}\ \ \text{for }e=(i,j)\in E(G),
		\]
		where $(E_{n,i})$ and $(U_{n,e})$ are stationary ergodic energy/edge-weight fields.
		In this case the diagonal subcocycle is explicit and
		\[
		\beta_1=\max_{1\le i\le d}\ \E\big(\log D_1(i,i)\big)
		= -\theta \min_{1\le i\le d}\ \E(E_{1,i}),
		\]
		so Theorem~\ref{th-shapes}(a) bounds the quenched free energy by a simple ``energy--entropy'' upper bound (energy from $\beta_1$, entropy from the $\log|E(G)|$ term).
		\par
		\noindent
		To make this interpretation more concrete, note that the expansion of $X_n=A_n\cdots A_1$ into monomials naturally enumerates \emph{time-ordered paths} in the state space $[d]$.
		Indeed, if we work in a basis where all entries are nonnegative (or simply take absolute values entrywise), then for each coordinate $j$ one can view $(X_n \odin)_j$ as a sum over sequences
		\[
		i_0 \to i_1 \to \cdots \to i_n=j
		\]
		where each step either stays put (a diagonal choice from $D_t$) or follows an allowed directed edge of $G$ (a choice from some $B_{t,e}$).
		The weight of a path is a product of the corresponding random factors along the path, so the total mass is a partition-function-like sum of random path weights.
		The top Lyapunov exponent $\gamma_1(\cala)$ is then the quenched exponential growth rate of this partition function.
		
		Theorem~\ref{th-shapes}(a) can be read as follows: in an acyclic environment, the long-run exponential growth is controlled by the ``best'' loop (the diagonal cocycle, contributing $\beta_1$), while the off-diagonal moves contribute only through the number of admissible local choices.
		The crude entropy term $\log|E(G)|$ comes from bounding the number of admissible label sequences by $|E(G)|^n$; in many applications this can be sharpened by replacing $|E(G)|$ with a smaller quantity such as the maximal out-degree or the spectral radius of an associated transition matrix.

		Fix $d\ge2$ and let $G$ be a directed acyclic graph (DAG) on vertex set $[d]$.  For each directed edge $e=(i,j)\in E(G)$, let $L_e:=E_{ij}$ be the $d\times d$ matrix with a single $1$ in entry $(i,j)$ and zeros elsewhere, and let $L_1:=I_d$.  Then
		\[
		\call:=\{L_1\}\cup\{L_e: e\in E(G)\}
		\]
		is a shape set (supports are disjoint), with $k=1+|E(G)|$.  Consider any stationary ergodic cocycle $(A_n)_{n\ge1}$ of the form
		\[
		A_n = D_n+\sum_{e\in E(G)} B_{n,e},
		\qquad \shape{D_n}\vee L_1=L_1,\ \ \shape{B_{n,e}}\vee L_e=L_e,
		\]
		where $D_n$ is diagonal and each $B_{n,e}$ is supported on the single entry prescribed by $e$.  If $D_n$ is a.s.\ invertible and $\E\log^+\|D_1^{\pm1}\|<\infty$, and if $\E\log^+\|B_{1,e}\|<\infty$ for all $e$, then Theorem~\ref{th-shapes} applies.
		
		The DAG assumption implies that the $\call$-shape graph has no directed cycles other than the self-loops induced by $L_1$ (products of edge labels correspond to directed paths, which cannot return in a DAG).  Hence $\call_{\calh}=\{1\}$ and $\beta=\beta_1$, where
		\[
		\beta_1=\lim_{n\to\infty}\frac1n\log\|D_nD_{n-1}\cdots D_1\|
		=\max_{1\le i\le d}\ \lim_{n\to\infty}\frac1n\sum_{t=1}^n \log|D_t(i,i)|\qquad\text{a.s.}
		\]
		If moreover $O_d\in\calv$ (this holds, for instance, whenever $G$ has no directed path of length $d$), then Theorem~\ref{th-shapes}(a) gives
		\[
		\gamma_1(\cala)\le \beta_1+\log k_*=\beta_1+\log|E(G)| \qquad(\text{here }k_*=|E(G)|).
		\]
		In this family the additive term $\log|E(G)|$ may be interpreted as a coarse ``branching entropy'' that accounts for the number of allowed sparsity types at each multiplication step, while $\beta_1$ captures the genuine exponential growth coming from the diagonal subcocycle.
	\end{example}

	\begin{remark}
		\label{rem:entropy-refine}
		The proof of Theorem~\ref{th-shapes}(a) expands $X_n$ into monomials indexed by label sequences in $\{1,\ldots,k\}^n$ and then applies a union bound over the $k^n$ terms.
		This is where the additive $\log k$ (or $\log k_*$) originates.
		In applications one can often sharpen this step:
		
		\begin{itemize}
			\item If only a subset of labels are ever reachable after a few steps (equivalently, the shape graph quickly collapses), one can replace $k$ by the effective number of labels that actually contribute.
			\item In ``graph-like'' constructions, one can bound the number of admissible label sequences in the expansion underlying part~(a) by the number of directed walks in an auxiliary transition graph (equivalently, by the number of directed paths in the shape graph).
			If $M$ is the adjacency matrix of that transition graph, then the number of length-$n$ walks is at most $\mathbf 1^\top M^{n-1}\mathbf 1\le C\,\rho(M)^{\,n}$, where $\rho(M)$ is the spectral radius.
			Thus the additive entropy term can often be improved from $\log k$ to $\log\rho(M)$ (or more crudely to $\log\Delta^+$, where $\Delta^+$ is the maximal out-degree).
			In particular, if the transition graph is acyclic, this count is subexponential and the limiting ``entropy'' contribution vanishes.
		\end{itemize}
		
		We do not pursue these refinements here, but the shape-graph viewpoint makes them natural.
	\end{remark}
	
	\subsection{Proof of Theorem~\ref{th-shapes}}
	The acyclicity assumption implies that every time-ordered monomial either (i) stays within a self-loop vertex for long stretches (hence inherits the exponential rate $\beta$ from the corresponding loop cocycle) or (ii) eventually dies by hitting $O_d$.
	Since non-loop transitions cannot repeat indefinitely, their contribution is at most subexponential (controlled using the $\E\log^+\|A_{1,s}\|$ hypotheses).
	The only remaining exponential cost is the number of distinct nonzero monomials, which is bounded by $k_{\rm eff}^n$; when $O_d$ is reachable one can take $k_{\rm eff}=k_*$ as in Theorem~\ref{th-shapes}, yielding the corresponding additive $\log k_{\rm eff}$ term.
	
	\medskip
	We now record the key uniform estimate needed in the proof of part~(a).
	
	\begin{lemma}[Uniform bound for nonzero monomials]\label{lem:uniform_monomial_bound}
		Assume the hypotheses of Theorem~\ref{th-shapes}.  Then for every $\varepsilon>0$ there exists an almost surely finite random constant $K_\varepsilon$ such that for all $n\ge 1$ and every multi-index $(i_1,\ldots,i_n)\in[k]^n$ for which the monomial
		\[
		A_{1,n}(i_1,\ldots,i_n):=A_{n,i_n}A_{n-1,i_{n-1}}\cdots A_{1,i_1}
		\]
		is nonzero, we have
		\[
		\|A_{1,n}(i_1,\ldots,i_n)\|_1 \le K_\varepsilon\,\exp\big((\beta+\varepsilon)n\big).
		\]
	\end{lemma}
	\begin{proof}
		Fix $\varepsilon>0$ and set $m:=|\calv|$.  Let $\delta:=\varepsilon/(4m)$.
		
		Define the stationary sequence
		\[
		Y_n:=\max_{1\le j\le k}\log^+\|A_{n,j}\|_1.
		\]
		Assumption~(i) implies $\E(Y_1)<\infty$.  By Lemma~\ref{lem:bc_log} (applied to $(Y_n)$), there exists an almost surely finite random constant $C_\delta$ such that for all $n\ge 1$ and all $1\le j\le k$,
		\[
		\|A_{n,j}\|_1 \le C_\delta\,e^{\delta n}.
		\]
		(Indeed the above bound holds for all sufficiently large $n$, and we absorb the finitely many remaining $n$ into $C_\delta$.)
		
		Fix a self-loop label $s\in\call_\calh$.  By assumption~(ii), the sequence $(A_{n,s})_{n\ge1}$ is stationary and ergodic, and $A_{n,s}$ is almost surely invertible with
		$\E\log^+\|A_{1,s}\|_1+\E\log^+\|A_{1,s}^{-1}\|_1<\infty$.
		Realize this stationary process as a linear cocycle over an ergodic measure-preserving transformation $\theta$ (so that $A_{n,s}(\omega)=A_{1,s}(\theta^{n-1}\omega)$).
		
		By Oseledets' multiplicative ergodic theorem for integrable linear cocycles (see e.g. \cite{arnold,boug}), for each $\delta>0$ there exists a measurable \emph{tempered} random variable $\widetilde K_{s,\delta}(\omega)<\infty$ such that for all $\ell\ge1$,
		\[
		\big\|A_{\ell,s}(\omega)\cdots A_{1,s}(\omega)\big\|_1\le \widetilde K_{s,\delta}(\omega)\,e^{(\beta_s+\delta)\ell}.
		\]
		Here ``tempered'' means that $\lim_{n\to\infty}\frac1n\log \widetilde K_{s,\delta}(\theta^n\omega)=0$ almost surely.
		In particular, $\sup_{t\ge1}\widetilde K_{s,\delta}(\theta^{t-1}\omega)e^{-\delta t}<\infty$ almost surely; define
		\[
		K_{s,\delta}(\omega):=\sup_{t\ge1}\widetilde K_{s,\delta}(\theta^{t-1}\omega)e^{-\delta t}.
		\]
		Then for all $t\ge1$ and all $\ell\ge1$,
		\begin{equation}\label{eq:loop_block_bound}
			\Big\|A_{t+\ell-1,s}A_{t+\ell-2,s}\cdots A_{t,s}\Big\|_1
			\le K_{s,\delta}(\omega)\,e^{\delta t}\,e^{(\beta_s+\delta)\ell}.
		\end{equation}
		Let $K_{\delta}:=\max_{s\in\call_\calh}K_{s,\delta}$, which is finite almost surely since $\call_\calh$ is finite.
		
		Fix $n$ and a multi-index $(i_1,\ldots,i_n)$ with $A_{1,n}(i_1,\ldots,i_n)\neq 0$.
		The associated walk on the shape graph visits at most $m$ vertices because the graph has no directed cycles other than self-loops.
		Consequently the index set $\{1,\ldots,n\}$ can be partitioned into at most $m$ contiguous \emph{blocks} in which either
		(i) the label is a self-loop label at the current vertex, repeated for the duration of the block, or
		(ii) the block consists of a single non-loop transition (a single time step at which the walk changes vertex).
		
		Let $\ell_1,\ldots,\ell_r$ be the lengths of the self-loop blocks (so $\sum_{q=1}^r\ell_q\le n$ and $r\le m$), and let $t_q$ be the starting time index of the $q$th self-loop block (so $t_q\le n$).
		Using \eqref{eq:loop_block_bound} for each loop block and the bound above for each of the at most $m$ transition steps, we obtain
		\[
		\|A_{1,n}(i_1,\ldots,i_n)\|_1
		\le (C_\delta)^m (K_\delta)^m \,
		\exp\!\Big(\delta m n\Big)\,
		\exp\!\Big(\sum_{q=1}^r (\beta_{s_q}+\delta)\ell_q\Big)\,
		\exp\!\Big(\delta m n\Big),
		\]
		where $s_q$ denotes the self-loop label used in block $q$.
		Since $\beta_{s_q}\le \beta$ and $\sum_{q=1}^r\ell_q\le n$, the middle exponential is bounded by $\exp((\beta+\delta)n)$.
		Absorbing constants, we find
		\[
		\|A_{1,n}(i_1,\ldots,i_n)\|_1
		\le (C_\delta K_\delta)^m \exp\!\big((\beta+\delta+2\delta m)n\big).
		\]
		With our choice $\delta=\varepsilon/(4m)$ we have $\delta+2\delta m\le \varepsilon$, so setting $K_\varepsilon:=(C_\delta K_\delta)^m$ yields
		\[
		\|A_{1,n}(i_1,\ldots,i_n)\|_1 \le K_\varepsilon\,e^{(\beta+\varepsilon)n}.
		\]
		This bound is uniform over all nonzero monomials of length $n$ and all $n\ge1$.
	\end{proof}

	\begin{proof}[Proof of (a)]
		Using the decomposition \eqref{A_ni_def} in $X_n$ and expanding the product, we write
		\beqn \label{Xn_expansion_1}
		X_n = \sum_{i_1=1}^{k} \cdots \sum_{i_n=1}^{k} A_{1,n}(i_1, \ldots, i_n), 
		\quad \text{where} \quad 
		A_{1,n}(i_1, \ldots, i_n) := A_{n,i_n}A_{n-1,i_{n-1}}\cdots A_{1,i_1}.
		\feqn
		Each multi-index $(i_1,\ldots,i_n)$ corresponds to a walk on the shape graph via the successive shapes
		$\shape{L_{i_1}},\, \shape{L_{i_2}L_{i_1}},\,\ldots,\,\shape{L_{i_n}\cdots L_{i_1}}$.
		If $O_d\in\calv$ and this walk ever enters the vertex $O_d$, then the corresponding monomial equals the zero matrix (its support becomes empty), hence contributes nothing to \eqref{Xn_expansion_1}.
		Therefore the sum \eqref{Xn_expansion_1} contains at most $k_{\rm eff}^n$ nonzero monomials, where we set
		\[
		k_{\rm eff}:=
		\begin{cases}
			k, & O_d\notin\calv,\\
			k_*, & O_d\in\calv,
		\end{cases}
		\]
		with $k_*$ as in the statement of the theorem.
		
		Fix $\varepsilon>0$.  By Lemma~\ref{lem:uniform_monomial_bound}, there exists an almost surely finite random constant $K_\varepsilon$ such that for all $n\ge1$ and every nonzero monomial $A_{1,n}(i_1,\ldots,i_n)$ appearing in \eqref{Xn_expansion_1}$,$
		\[
		\| A_{1,n}(i_1, \ldots, i_n) \|_1 \leq K_\varepsilon\, e^{(\beta+\varepsilon) n}.
		\]
		Consequently,
		\[
		\|X_n\|_1 \le k_{\rm eff}^n \, K_\varepsilon\, e^{(\beta+\varepsilon)n}.
		\]
		Taking logarithms, dividing by $n$, and letting $n\to\infty$ yields
		\[
		\gamma_1(\cala)\le \beta+\log k_{\rm eff}+\varepsilon.
		\]
		Since $\varepsilon>0$ is arbitrary, this gives $\gamma_1(\cala)\le \beta+\log k_{\rm eff}$, which is $\beta+\log k$ when $O_d\notin\calv$ and $\beta+\log k_*$ when $O_d\in\calv$.
	\end{proof}
	
	\medskip
	For $v\in\calv$, define the partial sum of monomials whose final shape is $v$ by
	\[
	Z_{n,v}:=\sum_{\substack{(i_1,\dots,i_n)\in[k]^n:\ \shape{L_{i_n}\cdots L_{i_1}}=v}} A_{1,n}(i_1, \ldots, i_n).
	\]
	Then $X_n=\sum_{v\in\calv} Z_{n,v}$.
	
	\begin{proof}[Proof of (b)]
		Fix $w\in\calw$, and let $s\in\call_\calw$ be the label of the self-loop at $w$.
		By the additional assumption in (b), there exists $r=r(w)\ge 1$ such that $\shape{L_s^n}=w$ for all $n\ge r$.
		For $n\ge r$ let $A_{1,n}^{(s)}:=A_{n,s}A_{n-1,s}\cdots A_{1,s}$; then $A_{1,n}^{(s)}$ is one of the monomials contributing to $Z_{n,w}$.
		
		By the assumption $v\wedge w=O_d$ for all $v\in\calv\setminus\{w\}$, the supports of $Z_{n,w}$ and $X_n-Z_{n,w}$ are disjoint, so
		\[
		\langle Z_{n,w},\, X_n-Z_{n,w}\rangle_F=0
		\quad\text{and hence}\quad
		\|X_n\|_F\ge \|Z_{n,w}\|_F.
		\]
		Under the additional hypothesis that the matrices $A_{n,s}$ are entrywise nonnegative a.s., the sum defining $Z_{n,w}$ is entrywise nonnegative and dominates each of its summands. In particular, $Z_{n,w}\ge A_{1,n}^{(s)}$ entrywise and therefore
		\[
		\|Z_{n,w}\|_F\ge \|A_{1,n}^{(s)}\|_F.
		\]
		Since all norms on $\calm_d$ are equivalent, the Lyapunov exponent $\beta_s$ is the same whether defined with $\|\cdot\|$ or $\|\cdot\|_F$, and thus
		\[
		\lim_{n\to\infty}\frac1n\log\|A_{1,n}^{(s)}\|_F=\beta_s.
		\]
		Combining the previous inequalities yields $\beta_s\le \gamma_1(\cala)$.  Taking the maximum over $s\in\call_\calw$ gives the claim.
	\end{proof}
	
	\section{Discussion and outlook}
	\label{sec:discussion}
	
	We close with a brief discussion highlighting how the shape-graph bounds should be interpreted and where one might seek refinements.
	
	\paragraph{Energy--entropy structure and regimes.} In the acyclic settings covered by Theorem~\ref{th-shapes}(a), the upper bound
	\[
	\gamma_1(\cala)\le \beta + \log k
	\qquad (\text{or } \beta+\log k_*\text{ when }O_d\in\calv)
	\]
	splits into two conceptually distinct pieces.
	The term $\beta$ is an ``energy'' contribution: it is the top Lyapunov exponent of the dominant loop cocycle(s), and it captures genuine exponential growth along directions that can persist indefinitely.
	The logarithmic term is an ``entropy'' contribution that accounts for the number of admissible monomials in the product expansion.
	In transfer-matrix language (Example above), $\beta$ corresponds to the best long-run local weight accumulation, while the entropy term corresponds to the proliferation of admissible time-ordered paths.
	
	In many models the entropy term is deliberately crude: if the off-diagonal components are small or rare, then $\gamma_1(\cala)$ is expected to be close to $\beta$, and the theorem quantifies that the discrepancy cannot exceed a logarithmic combinatorial correction.
	Conversely, when many off-diagonal components are simultaneously active and comparable in size to the loop components, the entropy term indicates that purely combinatorial proliferation can shift the top exponent by an $O(1)$ amount.
	
	\paragraph{Possible refinements of the entropy term.}
	As emphasized in Remark~\ref{rem:entropy-refine}, the $\log k$ term comes from counting \emph{all} label sequences.
	A natural refinement is to replace $k$ by an effective growth factor associated with the shape graph itself (e.g.\ a spectral radius, or a maximal branching factor after reachability constraints are enforced).
	Such refinements would make the bound sensitive to the geometry of the sparsity pattern, not only to the number of labels.
	
	\paragraph{Beyond acyclic shape graphs.}
	The self-loop-only hypothesis in Theorem~\ref{th-shapes} is tailored to situations where feedback through off-diagonal transitions is structurally impossible (e.g.\ DAG-supported sparsity).
	Allowing directed cycles would lead to genuinely new behavior: repeated off-diagonal transitions could create additional loop cocycles and may change the exponent in ways not captured by a single $\beta$.
	One possible direction is to analyze strongly connected components of the shape graph and to combine the present methods with subadditive/Kingman-type arguments on the induced component graph.
	
	\paragraph{Overlapping supports and cancellations.}
	A second limitation is the disjoint-support requirement in the definition of a shape set, which eliminates cancellations between distinct components.
	Relaxing this assumption is delicate: cancellations can reduce growth and can also make the expansion-based counting argument too pessimistic.
	It would be interesting to identify intermediate classes (e.g.\ sign-definite entries, monotone cones, or nonnegative matrices) where overlaps are allowed but the combinatorics remains tractable.
	
	\paragraph{Practical use.}
	From a user's perspective, the shape-graph method is most useful as a \emph{diagnostic bound}:
	given a structured model, it provides an explicit, checkable upper bound on the quenched growth rate and highlights which loop components control the exponent.
	Even when the bound is not tight, it can still be valuable for parameter scaling (e.g.\ dependence on temperature $\theta$ in transfer-matrix parametrizations) and for separating structural from probabilistic sources of growth.

	\section*{Acknowledgements}
	The author thanks Alex Roitershtein for helpful discussions and for comments on an earlier draft, in particular for his contribution to the proof of Proposition~\ref{thm101}.


\begin{thebibliography}{100}
		{\small 	
			
			\bibitem{adams}
			F.~C.~Adams, A.~M.~Bloch, and J.~C.~Lagarias,
			\emph{Random Hill’s equations, random walks, and products of random matrices},
			In: A.~Johann, H.~P.~Kruse, F.~Rupp, and S.~Schmitz (eds.),
			\emph{Recent Trends in Dynamical Systems},
			Springer Proceedings in Mathematics \& Statistics, vol. 35,
			Springer, 2013.
			\filbreak
			
			
			\bibitem{ahn}
			A.~Ahn and R.~Van Peski,
			\emph{Lyapunov exponents for truncated unitary and Ginibre matrices},
			Ann. inst. Henri Poincare (B) Probab. Stat. \textbf{59} (2023), 1029--1039).
			\filbreak
			
			\bibitem{aker}
			G.~Akemann, Z.~Burda, and M.~Kieburg,
			\emph{Universal distribution of Lyapunov exponents for products of Ginibre matrices},
			J. Phys. A: Math. Theor. \textbf{47} (2014), 395202.
			\filbreak
			
			
			\bibitem{rankone}
			J.~M.~Altschuler and P.~A.~Parrilo,
			\emph{Lyapunov exponent of rank-one matrices: ergodic formula and inapproximability of the optimal distribution},
			SIAM J. Control Optim. \textbf{58} (2020), 510--528.
			\filbreak
			
			\bibitem{arnold}
			L.~Arnold,
			\emph{Random Dynamical Systems},
			Springer, 1998.
			\filbreak
			
			\bibitem{barl}
			L.~Barreira,
			\emph{Lyapunov Exponents},
			Springer, 2017.
			\filbreak
			
			\bibitem{barr}
			L.~Barreira and C.~Valls,
			\emph{Lyapunov regularity via singular values},
			Trans. Am. Math. Soc. \textbf{369} (2017), 8409--8436.
			\filbreak
			
			\bibitem{boug}
			P.~Bougerol and J.~Lacroix,
			\emph{Products of random matrices with applications to Schr\"{o}dinger operators},
			Birkh\"{a}user, 1985.
			\filbreak
			
			
			\bibitem{cohen}
			J.~E.~Cohen and C.~M.~Newman,
			\emph{The stability of large random matrices and their products},
			Ann. Probab. \textbf{12} (1984), 283--310.
			\filbreak
			
			\bibitem{com}
			A.~Comtet, C.~Texier, and Y.~Tourigny,
			\emph{Products of random matrices and generalized quantum point scatters},
			J. Stat. Phys. \textbf{140} (2010), 427--466.
			\filbreak
			
			\bibitem{comt}
			A.~Comtet, J.~M.~Luck, C.~Texier, and Y.~Tourigny,
			\emph{The Lyapunov exponents of products of random $2\times 2$ matrices close to the identity},
			J. Stat. Phys. \textbf{150} (2013), 13--65.
			\filbreak
			
			
			\bibitem{blockas}
			D.~Dragi\u{c}evi\'{c} and K.~J.~Palmer,
			\emph{Dichotomies for triangular systems via admissibility},
			to appear in J. Dyn. Diff. Equat. (2024).
			\filbreak
			
			\bibitem{elf}
			M.~Embree and L.~N.~Trefethen,
			\emph{Growth and decay of random Fibonacci sequences},
			Proc. Royal Soc. A \textbf{455} (1987), 2471--2485.
			\filbreak
			
			
			\bibitem{fort}
			P.~J.~Forrester,
			\emph{Lyapunov exponents for products of complex Gaussian random matrices},
			J. Stat. Phys. \textbf{151}(2013), 796--808.
			\filbreak
			
			
			\bibitem{ens}
			P.~J.~Forrester and J.~Zhang,
			\emph{Lyapunov exponents for some isotropic random matrix ensembles},
			J. Stat. Phys \textbf{180} (2020), 558--575.
			\filbreak
			
			\bibitem{fuks}
			H.~Furstenberg and H.~Kesten,
			\emph{Products of random matrices},
			Ann. Math. Stat. \textbf{31} (1960), 457--469.
			\filbreak
			
			\bibitem{fifa}
			H.~Furstenberg and Y.~Kifer,
			\emph{Random matrix products and measures on projective spaces},
			Israel J. Math. \textbf{46} (1983), 12--32.
			\filbreak
			
			
			\bibitem{blockus}
			L.~Gerencs\'{e}r, G.~Michaletzky, and Z.~Orlovits,
			\emph{Stability of block-triangular stationary random matrices},
			Syst. Control Lett. \textbf{57} (2008), 620--625.
			\filbreak
			
			
			\bibitem{hen}
			H.~Hennion,
			\emph{Loi des grands nombres et perturbations pour des produits r\'{e}ductibles de matrices al\'{e}atoires ind\'{e}pendantes},
			Z. Wahrscheinlichkeitstheorie verw. Gebiete \textbf{67} (1984), 265--278.
			\filbreak
			
			\bibitem{horn}
			R.~G.~Horn and C.~R.~Johnson,
			\emph{Matrix Analysis}, 2nd edn.,
			Cambridge University Press, 2013.
			\filbreak
			
			\bibitem{fibbo}
			\'{E}.~Janvresse, B.~Rittaud, and T.~de~la~Rue,
			\emph{Almost-sure growth rate of generalized random Fibonacci sequences},
			Ann. Inst. H. Poincar\'{e} Probab. Statist. \textbf{46} (2010), 135--158.
			\filbreak
			
			\bibitem{kargin}
			V.~Kargin,
			\emph{On the largest Lyapunov exponent for products of Gaussian matrices},
			J. Sta.t Phys. \textbf{157} (2014), 70--83.
			\filbreak
			
			\bibitem{keyc}
			E.~Key,
			\emph{Computable examples of the maximal Lyapunov exponent}.
			Probab. Th. Rel. Fields \textbf{75} (1987), 97--107.
			\filbreak
			
			\bibitem{kiev}
			E.~S.~Key,
			\emph{Lyapunov exponents for matrices with invariant subspaces}.
			Ann. Probab. \textbf{16} (1988), 1721--1728.
			\filbreak
			
			\bibitem{kifer}
			Y.~Kifer, 
			\emph{Ergodic Theory of Random Transformations}, 
			Birkh\"{a}user, 1986.
			\filbreak
			
			\bibitem{letac}
			G.~Letac and V.~Seshadri,
			\emph{A characterization of the generalized inverse Gaussian distribution by continued fractions},
			Z. Wahrsch. verw. Gebiete \textbf{62} (1983), 485--489.
			\filbreak
			
			\bibitem{lima}
			R.~Lima and M.~Rahibe,
			\emph{Exact Lyapunov exponent for infinite products of random matrices},
			J. Phys. A: Math. Gen. \textbf{27} (1994), 3427.
			\filbreak
			
			\bibitem{darm}
			R.~Majumdar, P.~Mariano, H.~Panzo, L.~Peng, and A.~Sisti,
			\emph{Lyapunov exponent and variance in the CLT for products of random matrices related to random Fibonacci sequences},
			Discrete Contin. Dyn. Syst. Ser A \textbf{25} (2020), 4779--4799.
			\filbreak
			
			\bibitem{mann}
			D.~Mannion,
			\emph{Products of $2\times 2$ random matrices},
			Ann. Appl. Probab. \textbf{3} (1993), 1189--1218.
			\filbreak
			
			\bibitem{mark}
			J.~Marklof, Y.~Tourigny, and L.~Wo{\l}owski,
			\emph{Explicit invariant measures for products of random matrices},
			Trans. Amer. Math. Soc. \textbf{360} (2008), 3391--3427.
			\filbreak
			
			
			\bibitem{tom}
			T.~Mikosch and D.~Straumann,
			\emph{Stable limits of martingale transforms with application to the estimation of GARCH parameters},
			Ann. Statist. \textbf{34} (2006), 493--522.
			\filbreak
			
			
			
			\bibitem{newman}
			C.~M.~Newman,
			\emph{The distribution of Lyapunov exponents: exact results for random matrices},
			Commun. Math. Phys. \textbf{103} (1986), 121--126.
			\filbreak
			
			
			
			\bibitem{ots}
			N.~Otsuka and T.~Shimizu,
			\emph{Exponential stability of switched block triangular systems under arbitrary switching},
			Linear Multilinear Algebra \textbf{72} (2024), 655--677.
			\filbreak
			
			
			\bibitem{pinkus}
			S.~Pincus,
			\emph{Strong laws of large numbers for products of random matrices},
			Trans. Am. Math. Soc. \textbf{287} (1985), 65--89.
			\filbreak
			
			
			\bibitem{solvable}
			P.~Raj and D.~Pal,
			\emph{Lie-algebraic criterion for stability of switched differential-algebraic equations},
			IFAC-PapersOnLine \textbf{53} (2020), 2004--2009.
			\filbreak
			
			
			\bibitem{new}
			E.~Strickler,
			\emph{Randomly switched vector fields sharing a zero on a common invariant face},
			Stoch. Dynam. \textbf{21} (2021), 2150007.
			\filbreak
			
			
			\bibitem{viana}
			M.~Viana,
			\emph{Lectures on Lyapunov Exponents},
			Cambridge University Press, 2014.
			\filbreak
		}
		
		
	\end{thebibliography}
\end{document}